\documentclass[10pt,twoside, a4paper, english, reqno]{amsart}
\usepackage[dvips]{epsfig}
\usepackage{a4wide}
\usepackage{amscd}
\usepackage{amssymb}
\usepackage{amsthm}
\usepackage{amsmath}
\usepackage{latexsym}
\usepackage{graphicx,enumitem,dsfont}
\usepackage{subcaption}

\usepackage{upref}
\usepackage[colorlinks,backref]{hyperref}
\usepackage{color,cancel}
\usepackage[usenames,dvipsnames]{xcolor}
\usepackage{comment}
\usepackage{booktabs}
\usepackage{pdfsync}
\def\guy#1{\textcolor{red!100!black}{#1} }

\theoremstyle{plain}
\newtheorem{thm}{Theorem}[section]
\theoremstyle{plain}
\newtheorem{lemma}[thm]{Lemma}

\theoremstyle{definition}
\newtheorem{definition}{Definition}[section]
\newtheorem{remark}{Remark}[section]

\newtheorem*{maincorollary*}{Main Corollary}

\newenvironment{Assumptions}
{
\setcounter{enumi}{0}

\begin{enumerate}}
{\end{enumerate} }

\newcommand{\fr}{\mathfrak{L}_{\lambda}}
\newcommand{\norm}[1]{\left\|#1\right\|}
\newcommand{\abs}[1]{\left|#1\right|}
\newcommand{\sign}{\text{sign}}
\newcommand{\sgn}{\mathrm{sign}}

\newcommand{\Dx}{{\Delta x}}
\newcommand{\Dy}{{\Delta y}}
\newcommand{\Dt}{{\Delta t}}

\newcommand{\R}{\ensuremath{\mathbb{R}}}
\newcommand{\N}{\ensuremath{\mathbb{N}}}
\newcommand{\E}{\ensuremath{\mathbb{E}}}

\newcommand{\Div}{\mathrm{div}\,}
\newcommand{\rd}{\ensuremath{\mathbb{R}^d}}
\newcommand{\supp}{\ensuremath{\mathrm{supp}\,}}
\newcommand{\goto}{\ensuremath{\rightarrow}}
\newcommand{\grad}{\ensuremath{\nabla}}
\newcommand{\eps}{\ensuremath{\varepsilon}}
\newcommand{\Z}{\mathbb{Z}}

\def\N{{I\!\!N}}

\numberwithin{equation}{section} \allowdisplaybreaks

\title[Scheme for Stochastic fractional conservation laws]
{On the rate of convergence of a numerical scheme for Fractional conservation laws with noise}

\date{}


\author[Ujjwal Koley]{Ujjwal Koley}
\address[Ujjwal Koley]{\newline
	Centre for Applicable Mathematics,
	Tata Institute of Fundamental Research,
	P.O.\ Box 6503, GKVK Post Office,
	Bangalore 560065, India}
\email[]{ujjwal@tifrbng.res.in}

\author[Guy Vallet]{Guy Vallet}
\address[Guy Vallet]{\newline 
	LMAP UMR- CNRS 5142, IPRA BP 1155, 64013 Pau Cedex, France}
\email[]{guy.vallet@univ-pau.fr}

\begin{document}
\begin{abstract}
We consider a semi-discrete finite volume scheme for a degenerate fractional conservation laws driven by a cylindrical Wiener process. Making use of the bounded variation (BV) estimates, Young measure theory, and a clever adaptation of classical  Kru\v{z}kov theory, we provide estimates on the rate of convergence for approximate solutions to fractional problems. The main difficulty stems from the degenerate fractional operator, and requires a significant departure from the existing strategy to establish Kato’s type of inequality. Finally, as an application of this theory, we  demonstrate numerical convergence rates.
\end{abstract}

\maketitle
%\tableofcontents

\section{Introduction}
In this paper, we analyze finite volume schemes for a degenerate fractional conservation laws driven by a cylindrical Wiener process. To describe the problem, let us assume that $\big(\Omega, \mathbb{P}, \mathbb{F}, \{\mathbb{F}_t\}_{t\ge 0} \big)$ is a filtered probability space satisfying 
 the usual hypothesis \textit{i.e.},  $\{\mathbb{F}_t\}_{t\ge 0}$ is a right-continuous filtration such that $\mathbb{F}_0$ 
 contains all the $\mathbb{P}$-null subsets of $(\Omega, \mathbb{F})$. We are interested in numerical approximations of $L^2(\R^d)$-valued predictable processes $u(t, \cdot)$ which satisfy the following Cauchy problem
  \begin{equation}
\label{eq:stoc_con_brown}
 \begin{cases} 
 du(t,x) - \Div f(u(t,x))\,dt + \mathfrak{L}_{\lambda}[A(u(t,\cdot))](x)\,dt
 =\sigma(u(t,x)) \,dW(t), & \quad (x,t) \in \Pi_T, \\
 u(0,x) = u_0(x), & \quad  x\in \R^d,
 \end{cases}
 \end{equation}
 where $\Pi_T=\R^d \times (0,T)$ with $T>0$ fixed. Here $u_0: \R^d \rightarrow \R$ is a given initial function, $f:\R \mapsto \R^d, A: \R \mapsto \R $ are given (sufficiently smooth) nonlinear functions. Moreover, $\fr[u]$ denotes the classical fractional Laplace operator of order $\lambda \in (0,1)$, and defined as
\begin{align*}
\fr[\psi](x) := d_{\lambda}\, \text{P.V.}\, \int_{|z|>0} \frac{\psi(x) -\psi(x+z)}{|z|^{d + 2 \lambda}} \,dz,
\end{align*}
for some constants $d_{\lambda}>0$, and a sufficiently smooth function $\psi$. The nonlinearity $A'$ is allowed to be zero on an interval, thus \eqref{eq:stoc_con_brown} is a strongly degenerate fractional problem.
Furthermore, $W(t)$ is a cylindrical Wiener process given by $W(t)= \sum_{k\ge 1} g_k \beta_k(t)$, where $(\beta_k)_{k\ge 1}$ are mutually independent real valued standard Wiener processes, and $(g_k)_{k\ge 1}$ is a complete orthonormal system in a separable Hilbert space $\mathbb{H}$. Finally, we consider the mapping $\sigma(w): \mathbb{H}\goto L^2(\R^d)$, for each $w$ in $L^2(\R^d)$,  defined by $\sigma(w)g_k= h_k(w(\cdot))$. For our purpose, we assume that $h_k$ is Lipschitz-continuous and define $\mathbb{G}^2(r):= \sum_{k\ge 1} h_k^2(r)$.

\subsection{Review of the existing literature}
First observe that, for the case $A= \sigma =0$, the equation \eqref{eq:stoc_con_brown} represents a well-known conservation law in $\R^d$.
We refer to the pioneer papers by Kru\v{z}kov \cite{kruzkov} and Vol'pert \cite{Volpert} for existence and uniqueness results related to scalar conservation laws. Numerical schemes for deterministic hyperbolic conservation laws have been studied in \cite{Oleinik}, \cite{Hartenetal}, \cite{CrandallMajda} as well as others.
Well-posedness theory for the deterministic counterpart of \eqref{eq:stoc_con_brown} (\textit{i.e.}, the case $\sigma=0$) has been well studied in literature, starting with the work by Alibaud \cite{Alibaud}, and Cifani \& Jakobsen \cite{CifaniJakobsen}. 
For the convergence/rate of converegnce of numerical schemes for deterministic degenerate conservation laws one can refer to works by Cifani \& Jakobsen \cite{Cifani}, Droniou \cite{Jerome}, Karlsen \textit{et al.} \cite{ujjwal}, and the references therein.

On the other hand, well-posedness theory for stochastic conservation laws (\textit{i.e.}, the case $A=0$) has been established by several authors, see \cite{BaVaWit, BaVaWit_JFA, Chen-karlsen_2012, Debussche}. Regarding the convergence/rate of convergence of numerical schemes for stochastic conservation laws, we mention the paper by Kroker and Rohde \cite{kroker}, authors of this paper and Majee \cite{KMV1}, and a series of papers by Bauzet \textit{et al.} \cite{bauzet2015}, \cite{bauzet-fluxsplitting}, \cite{bauzet-finvolume}. Moreover, degenerate stochastic equations has been studied by many authors, see \cite{BaVaWitParab, vovelle, KMV}. For the stochastic non-local equation \eqref{eq:stoc_con_brown}, existence $\&$ uniqueness results have been recently developed, by exploiting a new technical framework for the proof of uniqueness, in \cite{BhKoleyVa} (linear diffusion case), and \cite{BhKoleyVa1} (degenerate diffusion case). Moreover, an explicit continuous dependence estimate on the nonlinearities and a rate of convergence estimate for the stochastic vanishing viscosity method was also established in \cite{BhKoleyVa,BhKoleyVa1}. Given the fact that, a plethora of phenomena in physics and finance are modeled by equations of the form \eqref{eq:stoc_con_brown} (\textit{e.g.} fluid flows in porous media, and pricing derivative securities in financial markets), there is a pressing need for efficient numerical schemes (such as finite volume schemes) to approximate the underlying problem \eqref{eq:stoc_con_brown}. Note that, the challenges in dealing with numerical schemes for SPDEs like \eqref{eq:stoc_con_brown} are manifold, mainly due to the presence of the fractional operator $\&$ \emph{multiplicative} noise term in \eqref{eq:stoc_con_brown}. Indeed, one needs to borrow and merge ideas from numerical methods for SDEs and approximation methods for the underlying deterministic problems. This is of course easier said than done since one needs to successfully capture the noise-noise interaction as well. In the realm of well-posedness theory for SPDEs, noise-noise interaction term plays a pivotal role, for details see \cite{BhKoleyVa, MKS01, K1, K2, BisKoleyMaj,koley2013multilevel,Koley3}.

\subsection{Aims and scope of this paper}
Due to the non availaibility of the explicit solution for stochastic balance laws, we intend to construct a numerical scheme for approximating the solution of \eqref{eq:stoc_con_brown}. Note that, although there are large number of papers exploring convergence results for stochastic conservation laws, a rigorous theoretical study for stochastic fractional conservation laws is in a state of infancy. In fact, the specific question about deriving the convergence rate for the approximate solutions to stochastic non-local problems is virtually untouched. For stochastic conservation laws, the convergence rate for monotone methods is known to be $\Dx^{1/2}$, $\Dx$ being the discretization parameter. Keeping in mind that the rate of convergence estimate intimately related to the so-called Kato's type of inequality, the main difficulties in dealing with the specific stochastic non-local problem \eqref{eq:stoc_con_brown} can be summarized as follows:
%We remark that, due to the
%presence of the nonlinear fractional diffusion, the approach of the above mentioned result is not applicable. Indeed, main difficulties are visible from the following observations:
\begin{itemize}
\item [(1)] For the stochastic non-local equation, by virtue of It\^o's formula, one requires to work with a \emph{smooth approximations} of the usual Kru\v{z}kov's entropies. 
\item [(2)] The numerical analysis of the deterministic counterpart of \eqref{eq:stoc_con_brown} heavily relies on the following (Kato's type of) inequality
\begin{equation}\label{imp_10}
\sgn\Big(u(x)-v(y)\Big) \Big(\fr[A(v)](y)- \fr[A(u)](x)\Big) \le \fr\Big[|A(u) -A(v)|\Big](x,y),
\end{equation}
and the above inequality \eqref{imp_10} does not hold for approximations of the usual Kru\v{z}kov's entropies. 
\end{itemize}
In view of the above incompatibility, to establish stochastic Kato's  inequality (even for a semi-discrete scheme!) one has to look beyond the traditional approach of proving convergence of numerical schemes. Indeed, the main technical achievement of this paper stems from successful demonstration of the stochastic Kato's  inequality. The method of proof is new and requires significant changes in computing hierarchical limits with respect to various parameters involved in the proof leading upto Kato's  inequality. In a nutshell, we send the parameter $\delta \downarrow 0$ before sending the parameter $l \downarrow 0$, see Section~\ref{sec:main-thm} for a complete description of the main ingredients of our method. Inevitably, above changes in hierarchical limits has effects on all the terms involved in the entropy inequality, and appropriate changes are required to deal with them.

To summarize, we consider a semi-discrete finite difference scheme and  show that the expected value of the $L^1$-difference of the approximate numerical solution and the unique entropy solution to \eqref{eq:stoc_con_brown} converges at some rate depending on $\lambda$, which is in accordance with the rate for the deterministic fractional conservation law \cite{Cifani}. Finally, let us mention that we do not know how to derive rate of convergence estimates for a fully-discrete scheme for \eqref{eq:stoc_con_brown}, but we strongly believe that the technical achievement of this paper is absolutely essential to establish such a result.

We organize the details as follows. We first briefly recall the entropy framework, introduce the semi-discrete finite volume scheme, and state the main result towards the end of Section \ref{sec:pre}. Further in Section \ref{sec:apriori&convergence} we prove \textit{a priori} estimates for the approximate numerical solutions. Section \ref{sec:main-thm} deals with the proof of the main theorem. Finally, we conclude by presenting some numerical results in Section \ref{numerics}.

%%%%%%%%%%%%%%%%%%%%%%%%%%%%%%%%%%%%%%%%%%%%%%%%%%%
 \section{Mathematical Framework and Statement of the Main Result}
 \label{sec:pre}
Throughout this paper, by the letter $C$, we denote various generic constants which may change from line to line. Given a separable Banach space $Z$, let us denote by $N^2_w(0,T,Z)$ the space of square integrable predictable $Z$-valued processes (cf. \cite{PrevotRockner} p.28 for example). The Euclidean norm on $\R^d$ is denoted by $||\cdot||$ and  the BV semi-norm is by $|\cdot|_{BV(\R^d)}$. We remark that the space BV$(\R^d)$ consists of functions with bounded variation on $\R^d$, endowed with the semi-norm $|u|_{BV(\R^d)}= TV_{x}(u)$, where $TV_{x}(u)$ is the total variation of $u$ defined on $\R^d$. We also denote by $\sgn_0(x)=\frac{x}{|x|}$ if $x\neq 0$, and $0$ otherwise.

In the rest of the paper, we consider the following assumptions:
 \begin{Assumptions}
\item \label{A1} The initial function $u_0$ is a deterministic function in $L^2(\R^d)\cap L^\infty(\R^d) \cap BV(\R^d)$.

\item \label{A2}  The flux function $f=(f_1,f_2,\cdots, f_d): \R \mapsto \R^d$ is a Lipschitz continuous function with $f_k(0)=0$, for all $1 \le k \le d$.
\item \label{A3} The function $A: \R \mapsto \R$ is a non decreasing Lipschitz continuous function with $A(0)=0$.
%\item \label{A4} \textcolor{red}{Useless, no (has to be in A5 or consider a standard brownian motion)? : $\sigma: \R \mapsto \R$ is a Lipschitz  function with $\sigma(0)=0$ and we assume that there exists $M>0$ such that $\sigma(u)=0$ for all $|u|>M$.}
\item \label{A5} There exists $M>0$ such that $\sigma(u)=0$, for all $|u|>M$ and $h_k(0)=0$, for all $k\ge 1$. Moreover, there exists a constant $K > 0$  such that 
\begin{align*} 
\sum_{k\ge 1}\big| h_k(u)-h_k(v)\big|^2  \leq K |u-v|^2,\,\,\text{which implies }\,\, \mathbb{G}^2(u)= \sum_{k\ge 1} h_k^2(u)\le K\,|u|^2, ~\text{for all} \,\,u,v \in \R.
\end{align*}
\end{Assumptions}

%\textcolor{red}{\begin{remark}
%Even if we artificially increase the value of $M$, it will be assumed in the sequel that $M \geq \|u_0\|_{\infty}$ too.
%\end{remark}
%}

\subsection{Stochastic Entropy Formulation}
It is well-known that weak solutions to \eqref{eq:stoc_con_brown} need not be unique. Consequently, an \emph{entropy admissibility condition} must be imposed to single out the physically correct solution. To describe the entropy framework for \eqref{eq:stoc_con_brown}, we need to first split the non-local operator $\fr$ into two terms: for each $r>0$, we write $\fr[\psi] := \mathfrak{L}_{\lambda, r}[\psi] + \mathfrak{L}_{\lambda}^{r}[\psi]$, where
\begin{align*}
\mathfrak{L}_{\lambda, r}[\psi](x)&:= d_{\lambda}\, \text{P.V.}\, \int_{|z|\le r} \frac{\psi(x) -\psi(x+z)}{|z|^{d + 2 \lambda}} \,dz, \\
\mathfrak{L}_{\lambda}^{r}[\psi](x)&:= d_{\lambda}\,\int_{|z|> r} \frac{\psi(x) -\psi(x+z)}{|z|^{d + 2 \lambda}} \,dz.
\end{align*}
We shall also denote 
\begin{align*}
\mathfrak{L}_{\lambda, r}^s[\psi](x)&:= d_{\lambda}\,\int_{r\geq |z|> s} \frac{\psi(x) -\psi(x+z)}{|z|^{d + 2 \lambda}} \,dz, \text{ if }0<s<r; \quad \mathfrak{L}_{\lambda, r}^s[\psi](x)=0, \text{ if }s\geq r>0.
\end{align*}
%It is well-known that weak solutions may be discontinuous and they are not uniquely determined by their initial data. Consequently, an admissibility condition, so called {\em entropy condition}, must be imposed to single out the physically correct solution. Since the notion of entropy solution is built around the so called entropy flux pair, we begin with the definition of entropy flux pair.
%
%\begin{definition}[Entropy-entropy flux pair]
%	A pair $(\eta,\zeta) $ is called an entropy-entropy flux pair 
%	if $ \eta \in C^2(\R) $ and $\eta \ge0$, and 
%	$\zeta = (\zeta_1,\zeta_2,....\zeta_d):\R \mapsto\rd $ is a vector field satisfying
%	$\zeta'(r) = \eta'(r)f'(r)$ for all $r$. An entropy-entropy flux pair $(\eta,\zeta)$ is called 
%	convex if $ \eta^{\prime\prime}(\cdot) \ge 0$.  
%\end{definition}
%
%For a small positive number $\eps>0$ we look at the parabolic perturbation 
%\begin{align}
%du_\eps(t,x) -\eps \Delta u_\eps(t,x)\,dt & + \mathfrak{L}_{\lambda}[A(u_\eps(t, \cdot))](x)\,dt  - \mbox{div}_x f(u_\eps(t,x)) \,dt 
%= \sigma(u_\eps(t,x))\,dW(t) 
%%+ \int_{E} \eta(u_\eps(t,x);z)\widetilde{N}(dz,dt)
%, \label{eq:viscous-Brown} 
%\end{align}
%of \eqref{eq:stoc_con_brown} has a unique weak solution $u_\eps(t,x)$ with initial data $u_{\eps}(0,x)=u_0^{\eps}(x)\in H^1(\R^d)$, where $u_0^{\eps}$ converges to $u_0$ in $L^2(\R^d)$. We have the following definition of stochastic entropy solution from \cite{BhKoleyVa1}.
We now recall the notion of the entropy solution for \eqref{eq:stoc_con_brown} from \cite[Definition 1.2]{BhKoleyVa1}.
\begin{definition} [Stochastic Entropy Solution]
	\label{defi:stochentropsol}
	An element $u\in N^2_w(0,T, L^2(\R^d))$, with intial data $u_0 \in L^2(\R^d)$, is called a stochastic entropy solution of \eqref{eq:stoc_con_brown} if given a non-negative test function $\psi\in C_{c}^{1,2}([0,\infty )\times\R^d) $ and a regular convex entropy-entopy flux pair $(\eta,\zeta)$, the following inequality holds:
	\begin{align}
	&  \int_{\R^d} \eta(u_0(x))\psi(0,x)\,dx + \int_{\Pi_T} \left[ \eta(u(t,x)) \partial_t\psi(t,x) -  \grad \psi(t,x)\cdot \zeta(u(t,x)) \right]\,dx\,dt  \\
	& -\int_{\Pi_T} \Big[ \mathfrak{L}_{\lambda}^{r}[A(u(t,\cdot))](x)\, \psi(t,x)\, \eta'(u(t,x)) + A^\eta_k(u(t,x)) \,\mathfrak{L}_{\lambda, r}[\psi(t,\cdot)](x) \Big]\,dx\,dt \notag \\
	& 
	+ \sum_{k\ge 1}\int_{\Pi_T} h_k(u(t,x))\eta^\prime (u(t,x))\psi(t,x)\,d\beta_k(t)\,dx
	+ \frac{1}{2}\int_{\Pi_T}\mathbb{G}^2(u(t,x))\eta^{\prime\prime} (u(t,x))\psi(t,x)\,dx\,dt  
	\ge  0, \,\, \mathbb{P}-\text{a.s} \notag
	.\label{inq:entropy-solun}
	\end{align}
\end{definition}

\subsection{Finite Difference Scheme}
\label{Finite_Difference_Scheme} 

%\textcolor{red}{We need to say that this is from Cifani et al.}

We begin by introducing some notations needed to define the
semi-discrete finite difference scheme. Let $\Dx$ denote a small positive number that represents the
spatial discretization parameter of the numerical scheme. We set $x_\alpha=\alpha\Dx$, for $\alpha=(\alpha_1,\cdots,\alpha_d)\in \Z^d$, to denote the spatial mesh points. 
%Moreover, for any function $u =u(x,t)$ admitting point values, we write $u_\alpha(t) = u(x_\alpha, t)$. 
Furthermore, following Cifani \textit{et al.} \cite{Cifani}, let us introduce the spatial grid cells
\begin{align*}
R_0=  \big[-\Dx/2,\Dx/2\big)^d,\quad R_\alpha = [x_{\alpha_1-1/2}, x_{\alpha_1+1/2}) \times \cdots \times [x_{\alpha_d-1/2}, x_{\alpha_d+1/2}) = x_\alpha + R_0,
\end{align*}
where $x_{j\pm1/2} = x_j \pm \frac{\Dx}{2}$, for $j \in \Z$ and $\alpha=(\alpha_1,\alpha_2,\cdots,\alpha_d) \in \Z^d$. 
\\[0.2cm]
For any given $(u_\alpha) \in \R^{\Z^d}$, set $u_\Dx =\sum_{\alpha \in \Z^d} u_\alpha \mathds{1}_{R_\alpha}$ ($\mathds{1}_{A}$ denotes the characteristic function of the set
$A$) and note that $u_\Dx$ is measurable in $\R^d$. Moreover, $u_\Dx \in L^p(\R^d)$ if and only if $u_\Dx \in \ell^p(\Z^d)$ with
\begin{align*}
p<+\infty:&\quad \|u_\Dx\|^p_{L^p(\R^d)}=\int_{\R^d}|u_\Dx(x)|^pdx = \Dx^d \sum_{\alpha \in \Z^d} |u_\alpha|^p = \Dx^d \|(u_\alpha)\|_{\ell^p(\Z^d)}^p,
\\
p=+\infty:&\quad \|u_\Dx\|_{L^\infty(\R^d)}=\sup_{\alpha \in \Z^d} |u_\alpha| = \|(u_\alpha)\|_{\ell^\infty(\Z^d)}.
\end{align*}
Furthermore, the BV semi-norm for a lattice function $u_{\Dx}$ is defined as
\begin{equation*}
\abs{u_\Dx(t,\cdot)}_{BV} = \Dx^{d-1} \sum_{\alpha\in\Z^d} \sum_{i=1}^d\abs{u_{\alpha+e_i}(t)- u_\alpha(t)},
\end{equation*}
where $\{e_1,e_2,\cdots e_d\}$ denotes the standard basis of $\R^d$. Finally, we denote by $D_k^{\pm}$ the  discrete forward and backward differences in space, \textit{i.e.,} 
\begin{equation*}
  D^{\pm}_k u_\Dx = \sum_{\alpha \in \Z^d} D^{\pm}_k u_\alpha \mathds{1}_{R_\alpha} = \sum_{\alpha \in \Z^d} 
  \pm \frac{u_{\alpha \pm e_k}- u_\alpha}{\Dx} \mathds{1}_{R_\alpha},
\end{equation*}
where the same notation $D^{\pm}_k$ is used for $u_\Dx$ and $u_\alpha$, with a slight abuse of notation.
%The discrete summation-by-parts formula is given by
%\begin{align*}
%\sum \limits_{\alpha\in \Z^d} u_\alpha D_{\pm} v_\alpha = -  \sum \limits_{\alpha\in \Z^d} v_\alpha _{\mp} u_\alpha.
%\end{align*}
%

We now propose the following semi-discrete (in time) finite volume scheme 
approximating the solutions generated by the equation \eqref{eq:stoc_con_brown}   
%\\[0.1cm]
%Thanks to Assumptions \ref{A1} to \ref{A5}, the maximum principle holds and the unique entropy solution is bounded by $M$. 
%
%For a given initial data $u_0$, we define the initial grid function $u_\Dx(0)= \sum_{\alpha \in \Z^d}u_\alpha(0)\mathds{1}_{R_\alpha}$ by \eqref{eq:scheme_initial}. Then, 
%$u_\Dx(t)= \sum_{\alpha \in \Z^d}u_\alpha(t)\mathds{1}_{R_\alpha}$ where the sequence $(u_\alpha(t))_{\alpha\in \Z^d}$ is associate with the stochastic differential equation
\begin{align}
  du_\alpha(t) + \sum^d_{k=1}D^{-}_k F_k(u_\alpha, u_{\alpha+e_k})\,dt  + \hat{\mathfrak{L}}_{\lambda} [ A(u_\Dx(t,\cdot)) ]_\alpha\,dt& = \sigma(u_\alpha(t))\,dW(t),\,\,t>0, \, \alpha \in\Z^d,
  \label{eq:levy_stochconservation_ discrete_laws_1}  \\
  u_\alpha(0) &=\frac{1}{\Delta x^d}\int_{R_\alpha} u_0(x)\,dx, \, \alpha \in\Z^d, \label{eq:scheme_initial}
  \end{align}
where $F_i$ is a monotone numerical flux corresponding to $f_i$, for each $1\le i \le d$.
Here monotone is understood in the following sense: $ F_i$ is Lipschitz continuous from $\R^2$ to $\R$, $F_i(u,u)=f_i(u)$ for any real $u$, $a \mapsto F_i(a,b)$ is non-decreasing and $b \mapsto F_i(a,b)$ is non-increasing. Some classical examples of monotone flux include Engquist-Osher flux, Godunov flux, and modified Lax-Friedrichs flux. Next, we give details about the discretization of the non-local term.

\subsubsection{Approximation of the non-local operator}
Following the idea developed in \cite{Cifani}, for any bounded $u_\Dx$, we will use the following (two layer) discretization of the non-linear non-local term: for any $u_\Dx =\sum_{\alpha \in \Z^d} u_\alpha \mathds{1}_{R_\alpha}$, 
\begin{align}\label{approx}
\mathfrak{L}_{\lambda}[A(u)] \approx \hat{\mathfrak{L}}_{\lambda} {[ A(u_{\Delta x})]}=\sum_{\alpha \in \Z^d} \hat{\mathfrak{L}}_{\lambda} {[ A (u_{\Delta x}) ]}_\alpha \mathds{1}_{R_\alpha} := \frac{1}{\Dx^d}\sum_{\alpha,\beta \in \Z^d}G_{\alpha,\beta} A(u_\beta)\mathds{1}_{R_\alpha},
\end{align}
where 
\begin{align*}
\hat{\mathfrak{L}}_{\lambda}[A(u_\Dx)]_\alpha :=& \Dx^{-d} \int_{R_\alpha} \mathfrak{L}_{\lambda}^{\frac{\Dx}{2}}[A(u_\Dx)](x) \,dx = 
d_\lambda\Dx^{-d} \int_{R_\alpha}\Big[ 
\int_{|z|>\frac{\Dx}{2}} \frac{A(u_\Dx)(x)-A(u_\Dx)(x+z)}{|z|^{d+2\lambda}} dz \Big]dx \\
&=\sum_{\beta \in \Z^d} A(u_\beta) \Big[ \Dx^{-d} \int_{R_\alpha} d_\lambda \int_{|z|>\frac{\Dx}{2}} \frac{\mathds{1}_{R_\beta}(x)-\mathds{1}_{R_\beta}(x+z)}{|z|^{d+2\lambda}} dz  dx \Big] := \sum_{\beta \in \Z^d} \Dx^{-d}G_{\alpha,\beta} A(u_\beta). 
\end{align*}

\begin{remark}
\label{remark_imp} \hfill \\ 
By denoting $d\mu(z) = d_\lambda \frac{dz}{|z|^{d+2\lambda}}$ and, for any $\alpha,\beta \in \Z^d$,  $G_{\alpha,\beta}=\int_{R_\alpha} \int_{|z|>\frac{\Dx}{2}}[\mathds{1}_{R_\beta}(x)-\mathds{1}_{R_\beta}(x+z) d\mu(z)  dx$, it is straightforward to verify that
\begin{align*}
&G_{\alpha,\beta}= G_{\beta,\alpha}, \quad G_{\alpha,\beta} \le 0,\quad 0 \leq  -\sum_{\beta \neq \alpha}G_{\alpha,\beta}= -\sum_{\beta \neq \alpha}G_{\beta,\alpha} = G_{\alpha,\alpha} \leq \Dx^{d} \mu\Big(|z|>\frac{\Dx}{2}\Big),
\\
&\forall k=1, \cdots,d, \quad G_{\alpha,\beta} = G_{\alpha +e_k,\beta+e_k}.
\end{align*}
\end{remark}

For technical reasons, we need to split the non-local integral in two parts and introduce the following notations: For any $ u_{\Dx}=\sum_{\beta \in \Z^d} u_\beta \mathds{1}_{R_\beta}$, and all $r>\frac{\Dx}{2}$, we write $\hat{\mathfrak{L}}_{\lambda} {[  u_{\Dx}]}=\hat{\mathfrak{L}}_{\lambda,r} {[ u_{\Dx} ]}+\hat{\mathfrak{L}}_{\lambda}^r {[  u_{\Dx} ]}$, where 
\begin{align*}
&\quad \hat{\mathfrak{L}}_{\lambda,r} {[ u_{\Dx} ]}=\sum_{\alpha \in \Z^d} \hat{\mathfrak{L}}_{\lambda,r} {[ \overline u(t,\cdot) ]}_\alpha \mathds{1}_{R_\alpha} = \frac{1}{\Dx^d}\sum_{\alpha,\beta \in \Z^d} G_{\alpha,\beta,r} \,u_\beta \mathds{1}_{R_\alpha},
\\
\text{with}\,\, & 
%\hat{\mathfrak{L}}_{\lambda,r} {[ u_{\Dx} ]}_\alpha = \frac{1}{\Dx^d} \int_{R_\alpha} \mathfrak{L}_{\lambda,r}^{\frac{\Dx}{2}}[\overline u(t,\cdot)](x)\,dx =  \frac{1}{\Dx^d} \int_{R_\alpha}  \int_{\frac{\Dx}{2}<|z|\leq r}\overline u(t,x) - \overline u(t,x+z) d\mu(z)dx
%\\ 
%\text{and}\quad &  
G_{\alpha,\beta,r} := \int_{R_\alpha}\int_{\frac{\Dx}{2}<|z|\le r} \big(\mathds{1}_{R_\beta}(x) - \mathds{1}_{R_\beta}(x+z) \big) \,d\mu(z)\,dx.
\end{align*}

\begin{align*}
\text{And,} 
&\quad \hat{\mathfrak{L}}_{\lambda}^r {[  u_{\Dx}]}=\sum_{\alpha \in \Z^d} \hat{\mathfrak{L}}_{\lambda}^r {[ u_{\Dx} ]}_\alpha \mathds{1}_{R_\alpha} = \frac{1}{\Dx^d}\sum_{\alpha,\beta \in \Z^d} G^r_{\alpha,\beta} \,u_\beta \mathds{1}_{R_\alpha}, \hspace{6cm}
\\[-0.2cm]
&\hspace{6cm}\text{with}\quad
%\hat{\mathfrak{L}}_{\lambda}^r {[ \overline u(t,\cdot) ]}_\alpha = \frac{1}{\Dx^d} \int_{R_\alpha} \mathfrak{L}_{\lambda,r}[\overline u(t,\cdot)](x)\,dx =  \frac{1}{\Dx^d} \int_{R_\alpha}  \int_{|z|>r}\overline u(t,x) - \overline u(t,x+z) d\mu(z)dx
%\\ 
%\text{and}\quad &  
G^r_{\alpha,\beta} := \int_{R_\alpha}\int_{r<|z|} \big(\mathds{1}_{I_\beta}(x) - \mathds{1}_{I_\beta}(x+z) \big) \,d\mu(z)\,dx.
\end{align*}

\begin{remark}
\label{remark_imp_bis} 
It is easy to check that Remark~\ref{remark_imp} holds for both $G^r_{\alpha,\beta}$ and $G_{\alpha,\beta,r}$, with the fact that $G^r_{\alpha,\alpha} \leq \Dx^{d} \mu(r>|z|>\frac{\Dx}{2})$ and $G_{\alpha,\beta,r} \leq \Dx^{d} \mu(|z|>r)$, and  that  $G_{\alpha,\beta} = G^r_{\alpha,\beta} + G_{\alpha,\beta,r}$.
\end{remark}
\noindent In light of the above observations, we can recast the scheme \eqref{eq:levy_stochconservation_ discrete_laws_1} as
\begin{align} \label{scheme}
  du_\alpha + 
  \frac{1}{\Dx}\sum^d_{i=1} \Big[F_i(u_\alpha, u_{\alpha+e_i})-F_i(u_{\alpha-e_i}, u_{\alpha})\Big]\,dt +\frac{1}{\Dx^d}\sum_{\beta \in \Z^d}G_{\alpha,\beta} A(u_\beta)\,dt & = \sigma(u_\alpha)\,dW(t).  
  \end{align}
Thanks to the assumptions on the data $\eqref{A1}-\eqref{A5}$ and Remark~\ref{remark_imp}, the solvability of \eqref{scheme} follows from a classical argument of stochastic differential equations with Lipschitz non-linearities (see \textit{e.g.}  Pr\'evot and R\"ockner \cite[Sec. 4.1 p.55]{PrevotRockner}).

\subsection{Discrete entropy inequality}

With the help of the above scheme \eqref{scheme}, we can derive the discrete entropy inequality. To that context, let $(\eta,\zeta)$ be a $C^2$ entropy-entropy flux pair. Given a non-negative test function $\psi\in C_{c}^{1,2}([0,\infty )\times\R^d)$, consider its piecewise approximation $\psi_\alpha(t):= \frac{1}{\Dx}\int_{R_\alpha} \psi(t,x)\,dx$, and apply It\^{o}'s product
rule to $\eta(u_\alpha(t))\psi_\alpha(t)$ to yield, 
\begin{align*}
&\eta(u_\alpha(T))\, \psi_\alpha(T) -  \eta(u_\alpha(0))\, \psi_\alpha(0) \\
 =& \int_0^T \left[\eta(u_\alpha(t)) \,\partial_t \psi_\alpha(t) - \eta^\prime (u_\alpha(t))\left( \frac{1}{\Dx}\sum^d_{i=1} \Big[F_i(u_\alpha, u_{\alpha+e_i})-F_i(u_{\alpha-e_i}, u_{\alpha})\Big] + \hat{\mathfrak{L}}_{\lambda} {[  A(u_\Dx(t,\cdot)) ]}_\alpha\right) \psi_\alpha(t) \right]\,dt \notag \\
 & + \sum_{k\ge 1}\int_0^T g_k(u_\alpha(t))\eta^\prime (u_\alpha(t)) \psi_\alpha(t)\,d\beta_k(t)
 + \frac{1}{2}\int_0^T\mathbb{G}^2(u_\alpha(t))\eta^{\prime\prime} (u_\alpha(t)) \psi_\alpha(t)\,dt.
 \end{align*}
Denote by  $\overline\psi(t,x)=\sum\limits_{\alpha \in \Z^d}\psi_\alpha(t) \mathds{1}_{R_\alpha}(x)$.  
Then, integrating over $R_\alpha$ and summing over $\alpha \in \Z^d$ yields
\begin{align*}
& \int_{\R^d} \eta(u_{\Dx}(T,x))\, \overline\psi(T,x)\,dx -  \int_{\R^d} \eta(u_{\Dx}(0,x))\, \overline\psi(0,x)\,dx \\
 &=  \int_{\Pi_T} \eta(u_{\Dx}(t,x)) \,\partial_t \overline \psi(t,x) -  \mathfrak{L}_{\lambda}^{\frac{\Dx}{2}}[ A(u_{\Dx}(t, \cdot))] (x) \overline\psi(t,x)\, \eta'(u_{\Dx}(t,x))  \,dx\,dt \notag \\
& - \int_{\Pi_T} \eta^\prime (u_{\Dx}(t,x))\frac{1}{\Dx}\sum^d_{i=1} \Big[F_i(u_{\Dx}(t,x), u_{\Dx}(t,x+\Dx e_i))-F_i(u_{\Dx}(t,x-\Dx e_i), u_{\Dx}(t,x))\Big]  \overline\psi(t,x)\,dx\,dt \notag \\
 & + \sum_{k\ge 1}\int_{\Pi_T} g_k(u_{\Dx}(t,x))\eta^\prime (u_{\Dx}(t,x)) \overline\psi(t,x)\,d\beta_k(t)\,dx
 + \frac{1}{2}\int_{\Pi_T}\mathbb{G}^2(u_{\Dx}(t,x))\eta^{\prime\prime} (u_{\Dx}(t,x)) \overline\psi(t,x)\,dx\,dt.
 \end{align*}
To deal with the fractional term we follow \cite[Appendix A]{BhKoleyVa1} to notice that, for any $r > \frac{\Dx}{2}$, 
\begin{align*}
& \int_{\Pi_T}  \mathfrak{L}_{\lambda}^{\frac{\Dx}{2}}[A( u_{\Dx}(t, \cdot))] (x) \overline\psi(t,x)\, \eta'(u_{\Dx}(t,x))  \,dx\,dt 
\\&= \int^T_0\left( \int_{\{ \frac{\Dx}{2}<|x-y|\}}\frac{ (A(u_\Dx(t,x))-A(u_\Dx(t,y)))((\overline\psi\eta')(u_\Dx(t,x))-(\overline\psi\eta')(u_\Dx(t,y)))}{|x-y|^{d+2\lambda}}\,dy\,dx\right)\,dt\\
&=\int^T_0\left( \int_{\{ \frac{\Dx}{2}<|x-y|<r\}}\frac{ (A(u_\Dx(t,x))-A(u_\Dx(t,y)))((\overline\psi\eta')(u_\Dx(t,x))-(\overline\psi\eta')(u_\Dx(t,y)))}{|x-y|^{d+2\lambda}}\,dy\,dx\right)\,dt\\
&+\int^T_0\left( \int_{\{ r<|x-y|\}}\frac{ (A(u_\Dx(t,x))-A(u_\Dx(t,y)))((\overline\psi\eta')(u_\Dx(t,x))-(\overline\psi\eta')(u_\Dx(t,y)))}{|x-y|^{d+2\lambda}}\,dy\,dx\right)\,dt\\
&\ge \int_{\Pi_T} A^\eta_k(u_\Dx(t,x))\mathfrak{L}_{\lambda,r}^{\frac{\Dx}{2}}[\overline{\psi}(t,\cdot)](x) \,dx\,dt + \int_{\Pi_T} \mathfrak{L}_\lambda^r[A(u_\Dx(t,\cdot)](x) \overline{\psi}(t,x) \eta'(u_\Dx(t,x))\,dx\,dt.
\end{align*}

Therefore, the discrete entropy inequality is understood in the following sense: 
\begin{align}
\label{inq:entropy-solun-discr}
 &0 \leq \int_{\R^d} \eta(u_{\Dx}(0,x))\, \psi(0,x)\,dx  +   \int_{\Pi_T} \eta(u_{\Dx}(t,x)) \,\partial_t \psi(t,x)\,dx\,dt  \\
 &- \int_{\Pi_T} \eta^\prime (u_{\Dx}(t,x))\, \frac{1}{\Dx}\sum^d_{i=1} \Big[F_i(u_{\Dx}(t,x), u_{\Dx}(t,x+\Dx e_i))-F_i(u_{\Dx}(t,x-\Dx e_i), u_{\Dx}(t,x))\Big]  \psi(t,x)\,dx\,dt \notag \\
&- \int_{\Pi_T}  A^\eta_k(u_\eps(t,x)) \mathfrak{L}_{\lambda,r}^{\frac{\Dx}{2}}[ \overline\psi(t,\cdot)] (x) +  \mathfrak{L}_{\lambda}^{r}[ A(u_{\Dx}(t, \cdot))] (x) \overline\psi(t,x)\, \eta'(u_{\Dx}(t,x))  \,dx\,dt \notag 
\\ & + \sum_{k\ge 1}\int_{\Pi_T} g_k(u_{\Dx}(t,x))\eta^\prime (u_{\Dx}(t,x)) \psi(t,x)\,d\beta_k(t)\,dx
 + \frac{1}{2}\int_{\Pi_T}\mathbb{G}^2(u_{\Dx}(t,x))\eta^{\prime\prime} (u_{\Dx}(t,x)) \psi(t,x)\,dx\,dt. \notag
 \end{align}
 %%%%%%%%%%%%%%%%%%%%%%%%%%%%%%%%%%%%%%%%%%%%%%%%%%%%%%%%%%%%%%%%%%%%%%
%For later use, recall that the discrete
%$\ell^{\infty}(\R)$, $\ell^1(\R)$ and $\ell^p(\R)$ $(1<p<\infty)$ norms, and BV semi-norm for a lattice function $u_{\Dx}$ are defined
%respectively as
%\begin{equation*}
%  \begin{aligned}
%    & \norm{u_{\Dx}(t,\cdot)}_{\infty} = \sup_{\alpha \in \Z^d} \abs{u_\alpha(t)}, \quad
%     \norm{u_\Dx(t,\cdot)}_{1} = \Dx^d \sum_{\alpha\in\Z^d} \abs{u_\alpha(t)}, \\
%    & \norm{u_\Dx(t,\cdot)}_{p} = \mysqrt{-2}{6}{p}{\Dx^d \sum_{\alpha\in\Z^d} \abs{u_\alpha(t)}^p}, \quad
%    \abs{u_\Dx(t,\cdot)}_{BV} = \Dx^{d-1} \sum_{\alpha\in\Z^d} \sum_{i=1}^d\abs{u_{\alpha+e_i}(t)- u_\alpha(t)}.
%  \end{aligned}
%\end{equation*}
%  where $\seq{e_i}_{i=1,..,d}$ denotes the cartesian orthonormal basis of $\R^d$. 
%
%
%Thanks to the assumptions on the data, the solvability of \eqref{eq:levy_stochconservation_ discrete_laws_1} follows from a classical argument of stochastic differential equations with Lipschitz nonlinearities (see e.g  Pr\'evot and R\"ockner \cite[Sec. 4.1 p.55]{PrevotRockner}.). 

\noindent We finish this section by stating the main result of this article, and for a proof of this main theorem we refer to Section \ref{sec:main-thm}.
\begin{thm}\label{Main_Theorem}
(Main Theorem) Let the assumptions $\eqref{A1}-\eqref{A5}$ hold, and $u_\Dx$ denotes the approximate solution generated by the finite volume scheme \eqref{scheme}. Moreover, let $u$ denotes the unique BV entropy solution to the problem \eqref{eq:stoc_con_brown}. Then there exists a constant C, independent of $\Dx$, such that for all $t \in (0,T]$
\begin{equation} 
\E \left[ \int_{\R^d} | u_\Dx(t,x) -u(t,x)|\,dx \right] \le C 
\begin{cases}
\sqrt{\Dx}, & \text{for} ~~ \lambda < \frac12, \vspace{0.1cm} \\
\sqrt{\Dx}| \log \Dx|, & \text{for}~~ \lambda =\frac12, \vspace{0.1cm} \\
(\Dx)^{1-\lambda}, & \text{for}~~ \lambda > \frac12.
\end{cases}
\end{equation}
%provided the initial error satisfies $$\E\Big[ \int_{\R^d} \big|u_{\Dx}(0,x)-u_0(x)\big|
% \,dx\Big] \le C\sqrt{\Dx}.$$
\end{thm}
%%%%%%%%%%%%%%%%%%%%%%%%%%%%%%%%%%%%%%%%%%%%%%%%%%%%%%%%%%%%%%%%%%%%%%%%%%%%%%%%%%%%%%%%%%%%%%%%%%%%%%%%%%%%%%%%%%%%%%%%%

%%%%%%%%%%%%%%%%%%%%%%%%%%%%%%%%%%%%%%%%%%%%%%%%%%%%%%%%%%%%%%%%%%%%%%%%%%%%%%%%%%%%%%%%%%%%%%%%%%%%%%%%%%%%%%%%%%%%%%%%%

\section{\textit{A Priori} Estimates}
\label{sec:apriori&convergence}
This section is devoted to the derivation of \emph{a priori} estimates for the approximate solutions $u_{\Delta x}(t,x)$ under the usual assumptions.
 \subsection{Uniform Moment Estimates} 
As we mentioned earlier, to ensure the convergence of the sequence of approximate solutions, one needs to obtain uniform moment estimates on it. In what follows, we start with the following simple but useful lemma which is essentially a discrete version of the entropy inequality \eqref{defi:stochentropsol}.

%{\color{red} Following lemma is still true for nonlinear fractional case with obvious modifications.}

\begin{lemma}
\label{lem:cellentropyinequality}
Let $\eta$ be an even, $C^2(\R)$ convex function with a bounded second derivative. Let $u_\alpha(t)$ be the approximate solution generated by the finite volume scheme \eqref{eq:levy_stochconservation_ discrete_laws_1}. Then $u_\alpha(t)$ satisfies the following cell entropy inequality: for any $k \in \R$, 
\begin{align*}
& d\eta(u_\alpha(t)-k) + \frac{1}{\Dx}\sum_{\beta \in \Z^d} \eta^\prime(u_\alpha(t)-k) G_{\alpha,\beta} \,A(u_\beta(t))\,dt 
\\ &+  
\sum^d_{i=1} \frac{|\eta^\prime(u_\alpha(t)-k)|}{\Dx}\bigg\{\Big(F_i[u_\alpha(t)\top k,u_{\alpha+e_i}(t) \top k]-F_i[u_{\alpha-e_i}(t) \top k, u_\alpha(t) \top k]\Big)
\\ & \hspace{5cm}-\Big(F_i[u_\alpha(t)\bot k,u_{\alpha+e_i}(t) \bot k]-F_i[u_{\alpha-e_i}(t) \bot k, u_\alpha(t) \bot k]\Big)\bigg\}\,dt 
\\
& \le \sigma(u_\alpha(t))\eta^\prime(u_\alpha(t)-k)\,dW(t) + \frac{1}{2} \sigma^2(u_\alpha(t))\eta^{\prime\prime}(u_\alpha(t)-k)\,dt,
\end{align*}
for all $\alpha \in \mathbb{Z}^d$ and almost all $\omega \in \Omega$. Here $a \top b:= \max\{a,b\}$ and $a \bot b:= \min\{a,b\}$.
\end{lemma}

\begin{proof}
 A simple application of It\^{o}'s formula applied to $\eta(u_\alpha(t)-k)$, where
$u_\alpha(t)$ satisfies the semi-discrete finite volume scheme \eqref{eq:levy_stochconservation_ discrete_laws_1}, leads to
 \begin{align}  \label{ito-discreate}
&   d\eta(u_\alpha(t)-k) + \frac{1}{{\Delta x}} \eta^\prime(u_\alpha(t)-k)\sum^d_{i=1} \Big( F_i(u_{\alpha}(t),u_{\alpha+e_i}(t))-F_i(u_{\alpha-e_i}(t),u_\alpha(t)) \Big)\,dt 
\\&  + \frac{1}{\Dx}\sum_{\beta \in \Z^d} \eta^\prime(u_\alpha(t)-k) G_{\alpha,\beta} \,A(u_\beta(t))\,dt 
 = \sigma(u_\alpha(t))\eta^\prime(u_\alpha(t)-k)\,dW(t) + \frac{1}{2} \sigma^2(u_\alpha(t))\eta^{\prime\prime}(u_\alpha(t)-k)\,dt. \notag
 \end{align} 
To manipulate terms coming from the flux $F_i$, we first observe that, for each $i=1,2,\cdots, d$ and any reals $u$, $v$ and $w$, we have
\begin{align*}
\eta^\prime(u-k)\Big[F_i(u,v)-F_i(w,u)]=|\eta^\prime(u-k)|\, \sign_0(u-k)\Big[F_i(u,v)-F_i(w,u)\Big].
\end{align*}
To simplify the notations, we denote $A:=\sign_0(u-k)\Big[ F_i[u,v]-F_i[w,u]\Big]$. Then for $u>k$, since $F_i$ is non-decreasing with respect to its first argument and non-increasing with respect to its second one, 
\begin{align*}
A =& F_i[(u-k)^++k,(v-k)+k]-F_i[(w-k)+k,(u-k)^++k] 
\\ \geq& F_i[(u-k)^++k,(v-k)^++k]-F_i[(w-k)^++k,(u-k)^++k] = F_i[u\top k,v \top k]-F_i[w \top k, u \top k].
\end{align*}
Moreover, 
\begin{align*}
F_i[u\bot k,v \bot k]-F_i[w \bot k, u \bot k]=&F_i[-(u-k)^-+k,-(v-k)^-+k]-F_i[-(w-k)^-+k,-(u-k)^-+k] 
\\ =&
F_i[k,k-(v-k)^-]-F_i[k-(w-k)^-,k] \geq F_i[k,k]-F_i[k,k]=0.
\end{align*}
Therefore, we conclude that 
\begin{align}
\label{001}
|\eta^\prime(u-k)|A \geq |\eta^\prime(u-k)|\bigg\{\Big(F_i[u\top k,v \top k]-F_i[w \top k, u \top k]\Big)-\Big(F_i[u\bot k,v \bot k]-F_i[w \bot k, u \bot k]\Big)\bigg\}.
\end{align}
A similar calculation reveals that the above inequality \eqref{001} also holds for $u\le k$.
%if $u=k$, 
%\begin{align*}
%A=&\sign_0(u-k)\Big[ F_i[u,v]-F_i[w,u]\Big] =0.
%\end{align*}
%Moreover, 
%\begin{align*}
%&\Big(F_i[u\top k,v \top k]-F_i[w \top k, u \top k]\Big)-\Big(F_i[u\bot k,v \bot k]-F_i[w \bot k, u \bot k]\Big)
%\\
%=&
%\Big(F_i[k,v \top k]-F_i[w \top k, k]\Big)-\Big(F_i[k,v \bot k]-F_i[w \bot k, k]\Big)
%\\
%\leq &
%\Big(F_i[k,k]-F_i[k, k]\Big)-\Big(F_i[k,k]-F_i[k, k]\Big) = 0
%\end{align*}
%and 
%\begin{align*}
%|\eta^\prime(u-k)|A \geq |\eta^\prime(u-k)|\bigg\{\Big(F_i[u\top k,v \top k]-F_i[w \top k, u \top k]\Big)-\Big(F_i[u\bot k,v \bot k]-F_i[w \bot k, u \bot k]\Big)\bigg\}.
%\end{align*}
%
%if $u<k$, 
%\begin{align*}
%A =& - \Big(F_i[u \bot k,v]-F_i[w,u \bot k] \Big) \geq - \Big(F_i[u \bot k,v \bot k]-F_i[w \bot k,u \bot k] \Big).
%\end{align*}
%Moreover, 
%\begin{align*}
%F_i[u\top k,v \top k]-F_i[w \top k, u \top k]=&F_i[k,v \top k]-F_i[w,k \top k] 
%\leq
%F_i[k,k]-F_i[k,k] =0
%\end{align*}
%and 
%\begin{align*}
%|\eta^\prime(u-k)| A \geq |\eta^\prime(u-k)|\bigg\{\Big(F_i[u\top k,v \top k]-F_i[w \top k, u \top k]\Big)-\Big(F_i[u\bot k,v \bot k]-F_i[w \bot k, u \bot k]\Big)\bigg\}.
%\end{align*}
This concludes the proof of the lemma.
\end{proof}
Now we are ready to prove uniform moment estimates. In what follows, we first state and prove the following lemma:
\begin{lemma}
Let the assumptions $\eqref{A1}-\eqref{A5}$ hold, and $u_{\Delta x}(t,x)$ be the approximate solution generated by the semi-discrete finite volume scheme \eqref{eq:levy_stochconservation_ discrete_laws_1}. Then, we have
\begin{align}
 \sup_{{\Delta x} >0}\sup_{0\le t \le T} \E \Big[||u_{\Delta x}(t,\cdot)||_p^p\Big] & \leq M^{\frac{p-1}{p}}\|u_0\|^{\frac1p}_{L^1(\R^d)}, \,\, \text{for} ~~ p\in \mathbb{N},~ p\ge 1.\label{uni:moment}
 \\
\|u_{\Delta x}(\omega, t,\cdot)\|_{L^\infty(\R^d)} & \le M,\ \forall t \geq 0,\, \mathbb{P}-a.s.,
\label{imp}
\end{align}
where $M$ is defined in assumption \ref{A5}. 

\end{lemma}
\begin{proof}
For a given convex function as in Lemma \ref{lem:cellentropyinequality} and $k=0$, after taking the expectation, we have
% \begin{align}
% \label{eq:test1}
%&\eta(u_\alpha(t)) - \eta(u_\alpha(0))  + \frac{1}{\Dx} \int_0^t\sum_{\beta \in \Z^d} \eta^\prime(u_\alpha(s)) G_{\alpha,\beta} \,A(u_\beta(s))\,ds \notag
%\\ &+ \frac{1}{{\Delta x}} \sum^d_{i=1}\int_0^t |\eta^\prime(u_\alpha(s))|\bigg\{\Big(F_i[u^+_\alpha(s),u^+_{\alpha+e_i}(s)]-F_i[u^+_{\alpha-e_i}(s), u^+_\alpha(s)]\Big)  \notag
%\\ & \hspace{5cm}-\Big(F_i[-u^-_\alpha(s),-u^-_{\alpha+e_i}(s)]-F_i[-u^-_{\alpha-e_i}(s), -u^-_\alpha(s)]\Big)\bigg\}\,ds \notag \\
%\le & \int_0^t \sigma(u_\alpha(s))\eta^\prime(u_\alpha(s))\,dW(s) + \frac{1}{2} \int_0^t \sigma^2(u_\alpha(s))\eta^{\prime\prime}(u_\alpha(s))\,ds
%\end{align} 
%and taking the expectation of \eqref{eq:test1}, one gets that 
\begin{align}\label{toto}
&\E [\eta(u_\alpha(t))] + \frac{1}{\Dx} \int_0^t\sum_{\beta \in \Z^d} \E[\eta^\prime(u_\alpha(s)) G_{\alpha,\beta} \,A(u_\beta(s))]\,ds\notag
\\ &\qquad + \frac{1}{{\Delta x}} \sum^d_{i=1}\E \Big[\int_0^t |\eta^\prime(u_\alpha(s))|\bigg\{\Big(F_i[u^+_\alpha(s),u^+_{\alpha+e_i}(s)]-F_i[u^+_{\alpha-e_i}(s), u^+_\alpha(s)]\Big)  \notag
\\ & \hspace{4cm}-\Big(F_i[-u^-_\alpha(s),-u^-_{\alpha+e_i}(s)]-F_i[-u^-_{\alpha-e_i}(s), -u^-_\alpha(s)]\Big)\bigg\}\,ds \Big]\notag \\
& \hspace{2cm} \le \frac{1}{2} \int_0^t \E[\sigma^2(u_\alpha(s))\eta^{\prime\prime}(u_\alpha(s))]\,ds + \eta(u_\alpha(0)).
\end{align}
Then, observe that the regularity of $\eta$ can be relaxed from $C^2$ to $C^1$, with $\eta^\prime$ being Lipschitz-continuous (and $\eta'(0)=0$).  We leverage this observation to conclude
%Thus,  $\eta'$ is Lipschitz-continuous with $\eta'(0)=0$, and, a.s. 
\begin{align}\label{tutu}
&\Big[ \sum_{\alpha \in \Z^d} \sum_{\beta \in \Z^d} |\eta^\prime(u_\alpha)| |G_{\alpha,\beta}| \, |A(u_\beta)|\Big]^2  \nonumber
 =
\Big[\sum_{\beta \in \Z^d} |A(u_\beta)| \sum_{\alpha \in \Z^d}  |\eta^\prime(u_\alpha)| |G_{\alpha,\beta}|\Big]^2 
\\ \leq & 
\Big[\sum_{\beta \in \Z^d} |A(u_\beta)|^2\Big]\Big[\sum_{\beta \in \Z^d}\Big( \sum_{\alpha \in \Z^d}  |\eta^\prime(u_\alpha)| |G_{\alpha,\beta}|\Big)^2\Big]
\leq 
C(A,\eta)\|(u_\beta)\|_{\ell^2(\Z^d)}^2 \Big[\sum_{\beta \in \Z^d}\Big( \sum_{\alpha \in \Z^d}  |u_\alpha| |G_{\alpha,\beta}|\Big)^2\Big]  \nonumber
\\ \leq &
C(A,\eta)\|(u_\beta)\|_{\ell^2(\Z^d)}^2 \Big[\sum_{\alpha \in \Z^d}  |u_\alpha|^2\Big] \Big[\sum_{\beta \in \Z^d} \sum_{\alpha \in \Z^d}  |G_{\alpha,\beta}|^2\Big] < +\infty, \quad \mathbb{P}-\text{a.s}.
\end{align}
Moreover, making use of Remark~\ref{remark_imp} and \eqref{tutu}, one has that
\begin{align*}
% \lim_{M \to \infty}\hspace{-0.75cm}\sum\limits_{\alpha \in \Z^d, \|\alpha\|_\infty \leq M} \hspace{-0.75cm} \sum^{\beta \in \Z^d} \eta^\prime(u_\alpha) G_{\alpha,\beta} \,A(u_\beta)
%=
\sum_{\alpha \in \Z^d} \sum_{\beta \in \Z^d} \eta^\prime(u_\alpha) G_{\alpha,\beta} \,A(u_\beta)
&= \sum_{\alpha \in \Z^d} \sum_{\beta \in \Z^d} \eta^\prime(u_\alpha) G_{\beta,\alpha} \,(A(u_\beta) - A(u_\alpha)) \\
\sum_{\alpha \in \Z^d} \sum_{\beta \in \Z^d} \eta^\prime(u_\beta) G_{\alpha,\beta} \,A(u_\alpha)
&= \sum_{\alpha \in \Z^d} \sum_{\beta \in \Z^d} \eta^\prime(u_\beta) G_{\alpha,\beta} \,(A(u_\alpha) - A(u_\beta)) = - \sum_{\alpha,\beta \in \Z^d} \eta^\prime(u_\beta) G_{\alpha,\beta} \,(A(u_\beta) - A(u_\alpha)). 
\end{align*}
This implies that
\begin{align}\label{positif}
\sum_{\alpha \in \Z^d} \sum_{\beta \in \Z^d} \eta^\prime(u_\alpha) G_{\alpha,\beta} \,A(u_\beta)
= \frac12 \sum_{\alpha \in \Z^d} \sum_{\beta \in \Z^d} (\eta^\prime(u_\alpha) - \eta^\prime(u_\beta)) G_{\alpha,\beta} \,(A(u_\beta) - A(u_\alpha)) \ge 0.
\end{align}
Assuming that $\eta(0)=0$ and since $(u_\alpha) \in C([0,T],\ell^2(\Z^d))$, one gets
\begin{align}\label{tata}
&\sum_{\alpha\in \Z^d}\E[\eta(u_\alpha(t))] + \frac{1}{\Dx} \int_0^t\sum_{\alpha \in \Z^d}\sum_{\beta \in \Z^d} \E[\eta^\prime(u_\alpha(s)) G_{\alpha,\beta} \,A(u_\beta(s))]\,ds \notag
\\ & \qquad +
\frac{1}{{\Delta x}} \sum^d_{i=1}\sum_{\alpha \in \Z^d}\E \Big[\int_0^t |\eta^\prime(u_\alpha(s))|\bigg\{\Big(F_i[u^+_\alpha(s),u^+_{\alpha+e_i}(s)]-F_i[u^+_{\alpha-e_i}(s), u^+_\alpha(s)]\Big)  \notag
\\ & \hspace{4cm}-\Big(F_i[-u^-_\alpha(s),-u^-_{\alpha+e_i}(s)]-F_i[-u^-_{\alpha-e_i}(s), -u^-_\alpha(s)]\Big)\bigg\}\,ds \Big]\notag
\\ & \hspace{2cm} \le  \frac{1}{2} \int_0^t \sum_{\alpha\in \Z^d}\E[\sigma^2(u_\alpha(s))\eta^{\prime\prime}(u_\alpha(s))]\,ds + \sum_{\alpha\in \Z^d }\eta(u_\alpha(0)).
\end{align}
Consider, in \eqref{tata}, that $\eta=\eta_\delta$ is the convex even function such that $\eta_\delta(0)=0$ and $\eta_\delta^{\prime}(x)=\min(1,\frac{x}{\delta})$ for positive $x$. Noting that
\begin{align*}
\Big(F_i[0,u^+_{\alpha+e_i}(s)]-F_i[u^+_{\alpha-e_i}(s), 0]\Big) -\Big(F_i[0,-u^-_{\alpha+e_i}(s)]-F_i[-u^-_{\alpha-e_i}(s), 0]\Big) \leq 0,
\end{align*}
 passing to the limit $\delta \to 0$, one gets for any $t$,   
\begin{align*}
0 \geq &\sum_\alpha\E|u_\alpha(t)| - \sum_\alpha|u_\alpha(0)|
\\ + &
\frac{1}{{\Delta x}} \sum^d_{i=1}\sum_\alpha \E\int_0^t \bigg\{\Big(F_i[u^+_\alpha(s),u^+_{\alpha+e_i}(s)]-F_i[u^+_{\alpha-e_i}(s), u^+_\alpha(s)]\Big)
\\ & \hspace{5cm}-\Big(F_i[-u^-_\alpha(s),-u^-_{\alpha+e_i}(s)]-F_i[-u^-_{\alpha-e_i}(s), -u^-_\alpha(s)]\Big)\bigg\}\,ds.
\end{align*}
Multiplying the above inequality by $\Delta x^d$, we are left with
\begin{align*}
\Delta x^d  \sum_{\alpha\in \Z^d}    \E |u_\alpha(t)|  \le \Delta x^d  \sum_{\alpha\in \Z^d} |u_\alpha(0)|
 \text{ and } 
\sup_{{\Delta x} >0}\sup_{0\le t \le T} \E \Big[||u_{\Delta x}(t,\cdot)||_{L^1(\R^d)}\Big] \leq ||u_0||_{L^1(\R^d)}.
\end{align*}

In order to prove the maximum principle, let us assume for the moment that, for any $i\in \{1,...,d\}$ and any real $u$, $A(u)$ is replaced by $A(T_M(u))$ and $f_i(u)$ by $f_i(T_M(u))$ where $T_M$ denotes the classical truncation: $T_M(u)=\max (-M, \min(u,M))$; and that the numerical flux becomes $F_i(T_M(u),T_M(v))$. This is a monotone numerical flux associated with $f_i(T_M(u))$ and  everything that has been done so far remains valid. 
\\
Note that as soon as the $L^\infty$ estimate in \eqref{imp} will be proved with the perturbation $A(T_M(u))$, $f_i(T_M(u))$ and $F_i(T_M(u),T_M(v))$, then the corresponding solution will be a solution to \eqref{eq:levy_stochconservation_ discrete_laws_1}, and, as the latter is unique, the estimate will hold for the solution $(u_\alpha)$.

Under the conditions described above, apply  It\^{o} formula to $\eta(u_\alpha(t)-M)$, where $\eta((-\infty,0])=\{0\}$ and
$u_\alpha(t)$ satisfies the semi-discrete finite volume scheme \eqref{eq:levy_stochconservation_ discrete_laws_1}. 
\\
As $\eta$ is convex, $\eta'\geq 0$, one gets that
\begin{align*}
\sigma(u_\alpha(t))\eta^\prime(u_\alpha(t)-M)=0, \qquad\sigma^2(u_\alpha(t))\eta^{\prime\prime}(u_\alpha(t)-M)=0
\end{align*}
and

\begin{align*}
&\eta^\prime(u_\alpha(t)-M)\Big( F_i(u_{\alpha}(t),u_{\alpha+e_i}(t))-F_i(u_{\alpha-e_i}(t),u_\alpha(t)) \Big)
\\ =&
\eta^\prime(u_\alpha(t)-M)\Big( F_i(M,u_{\alpha+e_i}(t))-F_i(u_{\alpha-e_i}(t),M) \Big)
\\ \geq &
\eta^\prime(u_\alpha(t)-M)\Big( F_i(M,M\top u_{\alpha+e_i}(t))-F_i(M\top u_{\alpha-e_i}(t),M) \Big) 
= \eta^\prime(u_\alpha-M)\Big( F_i(M,M)-F_i(M,M) \Big)=0,
\end{align*}
so that 
 \begin{align*} 
&   d\eta(u_\alpha(t)-M)+ \frac{1}{\Dx}\sum_{i \in \Z} \eta^\prime(u_\alpha(t)-M) G_{\alpha,\beta} \,A(u_\beta(t))\,dt \leq 0.
 \end{align*} 
By argument similar to \eqref{positif}, $\eta(u_\alpha(t)-M) \leq \eta(u_\alpha(0)-M)=0$ and $u_\alpha \leq M$. Moreover, applying It\^{o} formula to $\eta(u_\alpha(t)+M)$ implies that $u_\alpha \geq -M$.
%
%
%
%
%Apply now  It\^{o} formula to $\eta(u_\alpha(t)+M)$, where $\eta([0,+\infty))=\{0\}$. 
%\\
%Therefore, $\eta'\leq 0$ and one gets that
%\begin{align*}
%\sigma(u_\alpha(t))\eta^\prime(u_\alpha(t)+M)=0 \text{ and }\sigma^2(u_\alpha(t))\eta^{\prime\prime}(u_\alpha(t)+M)=0,
%\end{align*}
%\begin{align*}
%&\eta^\prime(u_\alpha(t)+M)\Big( F_i(u_{\alpha}(t),u_{\alpha+e_i}(t))-F_i(u_{\alpha-e_i}(t),u_\alpha(t)) \Big)
%\\ =&
%\eta^\prime(u_\alpha(t)+M)\Big( F_i(-M,u_{\alpha+e_i}(t))-F_i(u_{\alpha-e_i}(t),-M) \Big)
%\\ \geq &
%\eta^\prime(u_\alpha(t)+M)\Big( F_i(-M,-M\bot u_{\alpha+e_i}(t))-F_i(-M\bot u_{\alpha-e_i}(t),-M) \Big) 
%\\ =& \eta^\prime(u_\alpha(t)+M)\Big( F_i(-M,-M)-F_i(-M,-M) \Big)=0.
%\end{align*}
%Taking into account the non-local term  as done above, we have  $\eta(u_\alpha(t)+M) \leq \eta(u_\alpha(0)+M)=0$ and $u_\alpha \geq -M$.
%
%
Finally, assuming $p>1$, the result is proved by using an argument of interpolation.

\end{proof}

\subsection{Spatial Bounded Variation} 
Like its deterministic counterpart, we derive spatial BV bound for the approximate solutions under usual assumptions.

\begin{lemma}\label{lem:bv-estimate}%%%%%%%%%    BV estimate
Let the assumptions be true. Let $u_{\Delta x}(t,x)$ be the finite volume approximations prescribed by the finite difference scheme \eqref{eq:levy_stochconservation_ discrete_laws_1}. Then for any $t>0$
 	 \begin{align}
 	 \E\Big[|u_{\Delta x}(t,\cdot)|_{BV(\R^d)}\Big] \le C \,\E\Big[|u_0(\cdot)|_{BV(\R^d)}\Big].\label{prop:bv}
 	 \end{align}	
 \end{lemma}
 \begin{proof}
One has
 \begin{align*}
du_\alpha +& \frac{1}{\Dx}\sum^d_{i=1}[F_i(u_\alpha, u_{\alpha+e_i})-F_i(u_{\alpha-e_i}, u_{\alpha})]\,dt  + \frac{1}{\Dx^{d}} \sum_{\beta \in \Z^d} G_{\alpha,\beta}A(u_\beta) \,dt = \sigma(u_\alpha)\,dW(t),
\\
du_{\alpha+e_j} +& \frac{1}{\Dx}\sum^d_{i=1}[F_i(u_{\alpha+e_j}, u_{\alpha+e_i+e_j})-F_i(u_{\alpha-e_i+e_j}, u_{\alpha+e_j})]\,dt  + \frac{1}{\Dx^{d}} \sum_{\beta \in \Z^d} G_{\alpha+e_j,\beta}A(u_\beta) \,dt = \sigma(u_{\alpha+e_j})\,dW(t),
\end{align*}
and one is interested, \textit{via} It\^o's formula, to estimate $\eta(u_{\alpha+e_j}-u_{\alpha})$ and more precisely, the limit when $\eta$ converges to the absolution value function. 

Note, by Remark \ref{remark_imp}, that 
\begin{align*}
\sum_{\beta \in \Z^d} G_{\alpha+e_j,\beta}A(u_\beta) - \sum_{\beta \in \Z^d} G_{\alpha,\beta}A(u_\beta) = \sum_{\beta \in \Z^d} G_{\alpha,\beta-e_j}A(u_\beta) - \sum_{\beta \in \Z^d} G_{\alpha,\beta}A(u_\beta) = \sum_{\beta \in \Z^d} G_{\alpha,\beta}[A(u_\beta+e_j)-A(u_\beta)],
\end{align*}
so that, for a suitable regular approximation of $\eta=\eta_\delta$\footnote{the convex even function such that $\eta_\delta(0)=0$ and $\eta_\delta^{\prime}(x)=\min(1,\frac{x}{\delta})$ for positive $x$, $0 \leq \eta(u) \leq \frac1\delta u^2$ and $u^2\eta_\delta^{\prime\prime}(u) \leq \eta_\delta(u)$}, one has, for any $t$,  that 
\begin{align*}
&\E\eta(u_{\alpha+e_j}-u_{\alpha})(t) + \frac{1}{\Dx^d}\sum_{\beta \in \Z^d} \E\int_0^t \eta^\prime(u_{\alpha+e_j}-u_{\alpha})(s)G_{\alpha,\beta}[A(u_\beta+e_j)-A(u_\beta)](s) ds 
\\ =& 
-\sum^d_{i=1}\E\int_0^t\frac{\eta^\prime(u_{\alpha+e_j}-u_{\alpha})(s)}{\Dx}[F_i(u_{\alpha+e_j}, u_{\alpha+e_i+e_j})-F_i(u_{\alpha-e_i+e_j}, u_{\alpha+e_j})-F_i(u_\alpha, u_{\alpha+e_i})+F_i(u_{\alpha-e_i}, u_{\alpha})](s)\,ds
\\&+
\frac12 \E\int_0^t \eta^{\prime\prime}(u_{\alpha+e_j}-u_{\alpha})(s)[\sigma(u_{\alpha+e_j})-\sigma(u_{\alpha})](s)dy 
+ \E\eta(u_{\alpha+e_j}-u_{\alpha})(0).
\end{align*}
Using Lebesgue's theorem, 
\begin{align*}
\lim_{\delta \to 0}\E\int_0^t \eta^{\prime\prime}(u_{\alpha+e_j}-u_{\alpha})(s)[\sigma(u_{\alpha+e_j})-\sigma(u_{\alpha})](s)dy =0,
\end{align*}
and one is able to pass to the limit on $\eta_\delta$ and replace, above,  $\eta(u)$ by $|u|$ and $\eta^\prime(u)$ by sgn$_0(u)$. 

Note that, by adding and subtracting appropriate terms, we have 
\begin{align*}
& F_i(u_{\alpha+e_j}, u_{\alpha+e_i+e_j})-F_i(u_{\alpha-e_i+e_j}, u_{\alpha+e_j})-F_i(u_\alpha, u_{\alpha+e_i})+F_i(u_{\alpha-e_i}, u_{\alpha})
%\\ = &
%F_i(u_{\alpha+e_j}, u_{\alpha+e_i+e_j}) - F_i(u_{\alpha}, u_{\alpha+e_i+e_j}) + F_i(u_{\alpha}, u_{\alpha+e_i+e_j}) - F_i(u_\alpha, u_{\alpha+e_i})
%\\ &-\Big[ 
%F_i(u_{\alpha-e_i+e_j}, u_{\alpha+e_j}) - F_i(u_{\alpha-e_i+e_j}, u_{\alpha}) + F_i(u_{\alpha-e_i+e_j}, u_{\alpha}) - F_i(u_{\alpha-e_i}, u_{\alpha})
%\Big]
%\\ = &
%F_i(u_{\alpha+e_j}, u_{\alpha+e_i+e_j}) - F_i(u_{\alpha}, u_{\alpha+e_i+e_j}) + F_i(u_{\alpha}, u_{\alpha+e_i+e_j}) - F_i(u_\alpha, u_{\alpha+e_i})
%\\ &-\Big[ 
%F_i(u_{\alpha-e_i+e_j}, u_{\alpha+e_j}) - F_i(u_{\alpha-e_i}, u_{\alpha+e_j}) + F_i(u_{\alpha-e_i}, u_{\alpha+e_j}) - F_i(u_{\alpha-e_i}, u_{\alpha})
%\Big]
\\ &= 
\frac{F_i(u_{\alpha+e_j}, u_{\alpha+e_i+e_j}) - F_i(u_{\alpha}, u_{\alpha+e_i+e_j})}{u_{\alpha+e_j}-u_{\alpha}}(u_{\alpha+e_j}-u_{\alpha}) + \frac{F_i(u_{\alpha}, u_{\alpha+e_i+e_j}) - F_i(u_\alpha, u_{\alpha+e_i})}{u_{\alpha+e_i+e_j} - u_{\alpha+e_i}}(u_{\alpha+e_i+e_j} - u_{\alpha+e_i})
\\ &-\Big[ 
\frac{F_i(u_{\alpha-e_i+e_j}, u_{\alpha+e_j}) - F_i(u_{\alpha-e_i}, u_{\alpha+e_j})}{u_{\alpha-e_i+e_j}- u_{\alpha-e_i}}(u_{\alpha-e_i+e_j}- u_{\alpha-e_i}) + \frac{F_i(u_{\alpha-e_i}, u_{\alpha+e_j}) - F_i(u_{\alpha-e_i}, u_{\alpha})}{u_{\alpha+e_j}-u_{\alpha}}(u_{\alpha+e_j}-u_{\alpha})
\Big].
\end{align*}
Therefore, using  that $F_i$ is non-decreasing in its first argument and non-increasing in its second one, 
\begin{align*}
&\E\Big\{-\mathrm{sgn}_0(u_{\alpha+e_j}-u_{\alpha})\Big[ F_i(u_{\alpha+e_j}, u_{\alpha+e_i+e_j})-F_i(u_{\alpha-e_i+e_j}, u_{\alpha+e_j})-F_i(u_\alpha, u_{\alpha+e_i})+F_i(u_{\alpha-e_i}, u_{\alpha}) \Big] \Big\}
\\ \leq &
\E\Big\{-|\eta^\prime(u_{\alpha+e_j}-u_{\alpha})| \Big|\frac{F_i(u_{\alpha+e_j}, u_{\alpha+e_i+e_j}) - F_i(u_{\alpha}, u_{\alpha+e_i+e_j})}{u_{\alpha+e_j}-u_{\alpha}}\Big| |u_{\alpha+e_j}-u_{\alpha}| 
\\ &
+ |\eta^\prime(u_{\alpha+e_j}-u_{\alpha})| \Big|\frac{F_i(u_{\alpha}, u_{\alpha+e_i+e_j}) - F_i(u_\alpha, u_{\alpha+e_i})}{u_{\alpha+e_i+e_j} - u_{\alpha+e_i}}\Big| |u_{\alpha+e_i+e_j} - u_{\alpha+e_i}|
\\ &
+|\eta^\prime(u_{\alpha+e_j}-u_{\alpha})| \Big|\frac{F_i(u_{\alpha-e_i+e_j}, u_{\alpha+e_j}) - F_i(u_{\alpha-e_i}, u_{\alpha+e_j})}{u_{\alpha-e_i+e_j}- u_{\alpha-e_i}}\Big| |u_{\alpha-e_i+e_j}- u_{\alpha-e_i}| 
\\ &
-|\eta^\prime(u_{\alpha+e_j}-u_{\alpha})| \Big|\frac{F_i(u_{\alpha-e_i}, u_{\alpha+e_j}) - F_i(u_{\alpha-e_i}, u_{\alpha})}{u_{\alpha+e_j}-u_{\alpha}}\Big| |u_{\alpha+e_j}-u_{\alpha}| \Big\}
\\ =& \E\Big\{
\Big|\frac{F_i(u_{\alpha-e_i+e_j}, u_{\alpha+e_j}) - F_i(u_{\alpha-e_i}, u_{\alpha+e_j})}{u_{\alpha-e_i+e_j}- u_{\alpha-e_i}}\Big| |u_{\alpha-e_i+e_j}- u_{\alpha-e_i}| 
\\& \hspace{7cm}
-\frac{F_i(u_{\alpha+e_j}, u_{\alpha+e_i+e_j}) - F_i(u_{\alpha}, u_{\alpha+e_i+e_j})}{u_{\alpha+e_j}-u_{\alpha}}\Big| |u_{\alpha+e_j}-u_{\alpha}| \Big\}
\\ &
+ \E\Big\{\Big|\frac{F_i(u_{\alpha}, u_{\alpha+e_i+e_j}) - F_i(u_\alpha, u_{\alpha+e_i})}{u_{\alpha+e_i+e_j} - u_{\alpha+e_i}}\Big| |u_{\alpha+e_i+e_j} - u_{\alpha+e_i}|
\\&
\hspace{7.5cm}-\Big|\frac{F_i(u_{\alpha-e_i}, u_{\alpha+e_j}) - F_i(u_{\alpha-e_i}, u_{\alpha})}{u_{\alpha+e_j}-u_{\alpha}}\Big| |u_{\alpha+e_j}-u_{\alpha}|\Big\}.
\end{align*}
Remark finally that, as usual, the sum over $\alpha \in \Z^d$ of the above right-hand vanishes. 

Concerning the nonlocal part, by Remark \ref{remark_imp}
\begin{align*}
&\sum_{\beta \in \Z^d} \mathrm{sgn}_0(u_{\alpha+e_j}-u_{\alpha})G_{\alpha,\beta}[A(u_\beta+e_j)-A(u_\beta)]
\\ =& 
\mathrm{sgn}_0(u_{\alpha+e_j}-u_{\alpha})G_{\alpha,\alpha}[A(u_\alpha+e_j)-A(u_\alpha)]  +  \sum_{\beta \neq \alpha} \mathrm{sgn}_0(u_{\alpha+e_j}-u_{\alpha})G_{\alpha,\beta}[A(u_\beta+e_j)-A(u_\beta)]
\\ \geq& 
G_{\alpha,\alpha}|A(u_\alpha+e_j)-A(u_\alpha)|  +  \sum_{\beta \neq \alpha} G_{\alpha,\beta}|A(u_\beta+e_j)-A(u_\beta)| = \sum_{\beta \in \Z^d} G_{\alpha,\beta}|A(u_\beta+e_j)-A(u_\beta)|.
\end{align*}
Then, by arguments close to \eqref{tutu}, Fubini's theorem is applicable and 
\begin{align*}
&\E\int_0^t \sum_{\alpha \in \Z^d}\sum_{\beta \in \Z^d} \mathrm{sgn}_0(u_{\alpha+e_j}-u_{\alpha})(s)G_{\alpha,\beta}[A(u_\beta+e_j)-A(u_\beta)](s)ds 
\\ \geq& 
\E\int_0^t  \sum_{\alpha \in \Z^d}\sum_{\beta \in \Z^d} G_{\alpha,\beta}|A(u_\beta+e_j)-A(u_\beta)|(s) ds
= \E\int_0^t  \sum_{\beta \in \Z^d} \sum_{\alpha \in \Z^d}G_{\alpha,\beta}|A(u_\beta+e_j)-A(u_\beta)|(s) ds=0.
\end{align*}

In conclusion, after summing over $j$, we get
 \begin{align*}
%&\E\left[ \sum_{\alpha \in \Z^d} \big|u_{\alpha+e_j}(t)-u_\alpha(t)\big|\right] \le \E\left[ \sum_{\alpha \in \Z^d} \big|u_{\alpha+e_i}(0)-u_{\alpha}(0)\big|\right] 
%\\ \text{and} \quad &
%\\&
\E\left[ \sum^d_{j=1}\sum_{\alpha \in \Z^d} \big|u_{\alpha+e_j}(t)-u_\alpha(t)\big|\right] \le \E\left[\sum^d_{j=1} \sum_{\alpha \in \Z^d} \big|u_{\alpha+e_j}(0)-u_{\alpha}(0)\big|\right]. \\
 \end{align*}

 Note that, in view of the lower semi-continuity property and the positivity of the total variation $TV_x$,  $\E[TV_x(u)]$ makes sense for 
 any $u\in L^1(\Omega \times \R)$.
 Since $u_0 \in BV(\R)$, we conclude that, for all $t>0$
 \begin{align*}
  \E\Big[ TV_x(u_{\Delta x}(t))\Big] \le  \E\Big[ TV_x(u_0)\Big]. 
 \end{align*}
Again, since $ \E\Big[\|u_{\Delta x}(t,\cdot)\|_{L^1(\R)}\Big] \le  C \,\E\Big[ \|u_0(\cdot)\|_{L^1(\R^d)}\Big] $, we arrive at the following conclusion that the approximate solution $u_{\Delta x}(t,x)$ lies in the spatial BV class and satisfies \eqref{prop:bv}. This completes the proof.
 \end{proof}

\section{Proof of the Main Result: Theorem \ref{Main_Theorem}}
\label{sec:main-thm}

We start by introducing a special class of entropy functions, 
called convex approximation of absolute value function. To do so,  let $\eta:\R \rightarrow \R$ be a $C^\infty$ function satisfying 
\begin{align*}
\eta(0) = 0,\quad \eta(-r)= \eta(r),\quad 
\eta^\prime(-r) = -\eta^\prime(r),\quad \eta^{\prime\prime} \ge 0,
\end{align*} 
and 
\begin{align*}
\eta^\prime(r)=
\begin{cases} 
-1,\quad & \text{when} ~ r\le -1,\\
\in [-1,1], \quad & \text{when}~ |r|<1,\\
+1, \quad & \text{when} ~ r\ge 1.
\end{cases}
\end{align*} 
For any $\eps > 0$, define $\eta_\eps:\R \rightarrow \R$ by 
$\eta_\eps(r) = \eps \eta(\frac{r}{\eps})$. 
Then
\begin{align}\label{eq:approx to abosx}
|r|-L_1\eps \le \eta_\eps(r) \le |r|\quad 
\text{and} \quad |\eta_\eps^{\prime\prime}(r)| 
\le \frac{L_2}{\eps} {\bf 1}_{|r|\le \eps},
\end{align} 
where $L_1 := \sup_{|r|\le 1}\big | |r|-\eta(r)\big |$ and 
$L_2 := \sup_{|r|\le 1}|\eta^{\prime\prime} (r)|$. 

\noindent Moreover, for $\eta=\eta_\eps$, we define 
\begin{align*}
\begin{cases}
Q_k^\eta(a,b)= \int_b^a \eta^\prime(r-b) f_k^\prime(r)\,dr, \\
Q^\eta(a,b)=\big(f_1^\eta(a,b),f_2^\eta(a,b),\cdots,f_d^\eta(a,b)\big), \\
Q(a,b)= \sgn(a-b)(f(a)-f(b)) = f(a\top b)-f(a\bot b ).
\end{cases}
\end{align*}
For a small positive number $\eps>0$, we consider the following parabolic perturbation
of \eqref{eq:stoc_con_brown}
\begin{align*}
du_\eps(t,x) -\eps \Delta u_\eps(t,x)\,dt & + \mathfrak{L}_{\lambda}[A(u_\eps(t, \cdot))](x)\,dt  - \mbox{div}_x f(u_\eps(t,x)) \,dt 
= \sigma(u_\eps(t,x))\,dW(t). 
%+ \int_{E} \eta(u_\eps(t,x);z)\widetilde{N}(dz,dt)
\end{align*}
Following \cite{BhKoleyVa1}, we remark that it has a unique weak solution $u_\eps(t,x)$ with initial data $u_{\eps}(0,x)=u_0^{\eps}(x)\in H^1(\R^d)$, where $u_0^{\eps}$ converges to $u_0$ in $L^2(\R^d)$. 
Notice that $u_\eps \in H^1(\R^d)$, while for technical reasons we require higher regularity of $u_\eps$, therefore, we need to regularize $u_\eps$ by a space convolution. 
Let $\{\rho_\gamma\}_{\gamma}$ be a given mollifier-sequence in $\R^d$. Then following \cite{BhKoleyVa1} we observe that $u^\gamma_\eps:= u_\eps* \rho_\gamma$ satisfies
\begin{eqnarray*}
\partial_t \bigg[u_\eps*\rho_\gamma - \int_0^t \sigma(u_\eps)*\rho_\gamma  dW \bigg]- \big[\eps \Delta(u_\eps* \rho_\gamma)-\mathfrak{L}[A(u_\eps)*\rho_\gamma]+ \Div f(u_\eps)*\rho_\gamma \big]=0.
\end{eqnarray*}

Let $\rho$ and $\varrho$ be the standard nonnegative 
mollifiers on $\R$ and $\R^d$ respectively such that 
$\supp(\rho) \subset [-1,0]$ and $\supp(\varrho) = \overline{B_1}(0)$.  
We define  $\rho_{\delta_0}(r) = \frac{1}{\delta_0}\rho(\frac{r}{\delta_0})$ 
and $\varrho_{\delta}(x) = \frac{1}{\delta^d}\varrho(\frac{x}{\delta})$, 
where $\delta$ and $\delta_0$ are two positive parameters.  Given a  nonnegative test 
function $\psi\in C_c^{1,2}([0,\infty)\times \rd)$ and two 
positive constants $\delta$ and $ \delta_0 $, we define 
\begin{align}
\label{test_function}
\phi_{\delta,\delta_0}(t,x, s,y) = \rho_{\delta_0}(t-s) 
\,\varrho_{\delta}(x-y) \,\psi(t,x).
\end{align}
Clearly $ \rho_{\delta_0}(t-s) \neq 0$ only 
if $s-\delta_0 \le t\le s$ and hence $\varphi_{\delta,\delta_0}(t,x; s,y)= 0$, 
outside  $s-\delta_0 \le t\le s$. 
Moreover, let $\varsigma$ be the standard symmetric 
nonnegative mollifier on $\R$ with support in $[-1,1]$ 
and $\varsigma_l(r)= \frac{1}{l} \varsigma(\frac{r}{l})$, 
for $l > 0$.

We multiply the entropy inequality \eqref{inq:entropy-solun-discr} by $\varsigma_l(u_{\Dy}(s,y)-k)$, take the expectation and integrate with respect to $s,y$ and $k$ over $\Pi_T \times \R$ to get the following form,
\begin{align}
& 0\le \E \left[\int_{\R\times\Pi_T\times\R^d} \eta(u^{\gamma}_\eps(0,x)-k)\,
\phi_{\delta,\delta_0}(0,x,s,y) \,\varsigma_l(u_{\Dy}(s,y)-k) \,dx\,dy\,ds\,dk\right] \notag %I_1
\\&  
+ \E \left[\int_{\R \times \Pi_T\times\Pi_T} \eta(u^\gamma_\eps(t,x)-k)\partial_t \phi_{\delta,\delta_0}(t,x,s,y)\,
\varsigma_l(u_{\Dy}(s,y)-k)\,dx\,dt\,dy\,ds\,dk \right]\notag %I_2
\\ & 
+  \int_{\R \times \Pi_T} \E\left[\sum_{n\ge 1}\int_{\Pi_T}
  g_n(u^\gamma_\eps(t,x))\eta^\prime (u^\gamma_\eps(t,x)-k)\, \phi_{\delta,\delta_0}(t,x,s,y)\,d\beta_k(t) \,dx\,\right] \varsigma_l(u_{\Dy}(s,y)-k)\,dy\,ds\,dk  \notag  %I_3
 \\ &
 + \frac{1}{2}\, \E \left[ \int_{\R \times \Pi_T\times\Pi_T} \mathbb{G}^2(u^\gamma_\eps(t,x))\eta^{\prime\prime} (u^\gamma_\eps(t,x) -k)\, \phi_{\delta,\delta_0}(t,x,s,y)\, \varsigma_l(u_{\Dy}(s,y)-k)\,dx\,dt\,dy\,ds\,dk \right] \notag  %I_4
\\ & 
+\E \left[\int_{\R \times\Pi_T\times\Pi_T} 
 Q^\eta(u^\gamma_\eps(t,x),k) \cdot \grad_x  \phi_{\delta,\delta_0}(t,x;s,y)\, \varsigma_l(u_{\Dy}(s,y)-k)\,dx\,dt\,dy\,ds\,dk\right]  \notag  %I_5
\\& 
- \E \left[\int_{\R \times\Pi_T\times\Pi_T} \mathfrak{L}_{\lambda}^r[A(u_\eps)^\gamma(t,\cdot))](x)\, \phi_{\delta,\delta_0}(t,x,s,y)\, \eta'(u^\gamma_\eps(t,x) -k) \varsigma_l(u_{\Dy}(s,y)-k)\,dx\,dt 
\,dy\,ds\,dk \right] \notag %I_6
\\& 
-  \E \left[\int_{\R \times \Pi_T\times\Pi_T}  \,\mathfrak{L}_{\lambda,r} [A(u_\eps)^\gamma(t,\cdot)](x)\phi_{\delta,\delta_0}(t,x,s,y) \eta'(u^\gamma_\eps(s,y)-k) \, \varsigma_l(u_{\Dy}(s,y)-k)\,dx \,dt\,dy\,ds\,dk \right] \notag %I_7
\\ & 
-\eps \, \E \left[\int_{\R \times \Pi_T\times\Pi_T} \eta^{\prime}(u^\gamma_{\eps}(t,x)-k) \nabla_x u^\gamma_{\eps} (t,x) \cdot \nabla_x \phi_{\delta,\delta_0}(t,x;s,y)\, \varsigma_l(u_{\Dy}(s,y)-k)\,dx\,dt\,dy\,ds\,dk\right]\notag  %I_8
 \\[2mm]
& =:  I_1 + I_2 + I_3 +I_4 + I_5 + I_6 + I_7 + I_8. \label{stochas_entropy_1-levy-d}
\end{align}

The corresponding discrete version is obtained by multiplying the discrete entropy inequality \eqref{inq:entropy-solun-discr} by $\varsigma_l(u_\eps(t,x)-k)$, taking the expectation of the result and integrating with respect to $t,x$ and $k$ over $\Pi_T \times \R$ to get, 
\begin{align}
0\le  & \E \left[ \int_{\R\times \Pi_T \times\R^d} \eta \big(u_{\Dy}(0)-k\big)
\phi_{\delta,\delta_0}(t,x,0,y) \varsigma_l(u^\gamma_{\eps}(t,x)-k)\,dy \,dx\,dt\,dk\right] \notag %J_1
\\ & + 
\E \left[\int_{\R \times \Pi_T \times \Pi_T }  \eta(u_{\Dy}(s,y)-k)\partial_s \phi_{\delta,\delta_0}(t,x,s,y)
\varsigma_l(u^\gamma_{\eps}(t,x)-k)\,dy\,ds \,dx\,dt\,dk \right]\notag  %J_2
\\  & +  
 \int_{\R \times \Pi_T }\E\left[\sum_n \int_{\Pi_T} g_n(u_{\Dy}(s,y)) \eta^{\prime}(u_{\Dy}(s,y)-k) \phi_{\delta,\delta_0}(t,x,s,y) \,dy\,d\beta_k(s)\right]
 \,\varsigma_l(u^\gamma_{\eps}(t,x)-k)\,dt \,dx \,dk \notag  %J_3
 \\& + 
 \frac{1}{2}  \E \left[\int_{\R \times \Pi_T\times\Pi_T}  \mathbb{G}^2(u_{\Dy}(s,y)) \eta^{\prime\prime}(u_{\Dy}(s,y)-k)
\phi_{\delta,\delta_0}(t,x,s,y)\,\varsigma_l(u^\gamma_{\eps}(t,x)-k)  \,dy\,ds\,dt \,dx\,dk \right] \notag  %J_4
\\& 
- \E \bigg[\int_{\R \times\Pi_T\times\Pi_T}\hspace{-1cm}  \eta^{\prime}(u_{\Dy}(s,y)-k) \frac{1}{\Dy} \sum^d_{i=1}\Big[F_i(u_{\Dy}(s,y), u_{\Dy}(s,y+\Dy e_i))-F_i(u_{\Dy}(s,y-\Dy e_i), u_{\Dy}(s,y))\Big]
 \notag
  \\[-0.5cm] &\hspace{9 cm} \times
  \phi_{\delta,\delta_0}(t,x,s,y)  \varsigma_l(u^\gamma_{\eps}(t,x)-k)\,ds\,dy\,dx\,dt\,dk\bigg] \notag  %J_5
   \\  & 
   - \E \left[\int_{\R \times \Pi_T\times\Pi_T} \mathfrak{L}_{\lambda}^r[A(u_\Dy(s,\cdot))](y)\, \overline \phi^y_{\delta,\delta_0}(t,x,s,y)\, \eta'(u_\Dy(s,y) -k) \varsigma_l(u^\gamma_\eps(t,x)-k)\,dy\,ds 
\,dx\,dt\,dk \right] \notag  %J_6
\\ & 
-  \E \left[\int_{\R \times\Pi_T\times\Pi_T} A^\eta_k(u_{\Dy}(s,y)) \,\mathfrak{L}_{\lambda,r}^{\frac{\Dy}{2}} [\overline \phi^y_{\delta,\delta_0}(t,x,s,\cdot)](y)\varsigma_l(u^\gamma_\eps(t,x)-k)
\,dy\,ds\,dx \,dt\,dk \right] \notag  %J_7
\\& 
:= J_1 + J_2 + J_3 + J_4 + J_5 + J_6 + J_7, \label{stoc_entropy_2}
\end{align}
where $\overline \phi^y_{\delta,\delta_0}(t,x,s,y)=\sum_{\alpha \in \Z^d} \frac{1}{\Dy^d}\int_{R_\alpha} \phi_{\delta,\delta_0}(t,x,s,z)\; dz \mathds{1}_{R_\alpha}(y)$.
\\[0.2cm]
Our goal is to estimate the expected value of the $L^1$ difference between $u_{\Dy}$ and $u_{\eps}$ in terms of the small parameters $\delta, \Dy$ and $r$ ($r>\frac{\Dy}{2}$), which are sufficiently small but fixed.

%%%%%%%%%%%%%%%%%%%%%%%%%
\begin{comment}
\subsection{$I_1+J_1$}

\begin{align*}
I_1+J_1 = &\E \Big[\int\limits_{\R^d\times\Pi_T\times\R} \beta(u^{\eps}_0(x)-k)\,
\rho_{\delta_0}(-s) \,\varrho_{\delta}(x-y)\, \psi(0,x) \,\varsigma_l(u_\Dx(s,y)-k) \,dx\,dy\,ds\,dk\Big]
\\ &+
\E \Big[ \int\limits_{\R^d\times \Pi_T \times\R} \beta \big(u_{\Delta x}(0)-k\big)
\rho_{\delta_0}(t) \,\varrho_{\delta}(x-y)\, \psi(t,x) \varsigma_l(u_{\eps}(t,x)-k)\,dy \,dx\,dt\,dk\Big]
\\ = & 
\E \Big[ \int\limits_{\R^d\times \Pi_T \times\R} \beta \big(u_{\Delta x}(0)-u_{\eps}(t,x)+k\big)
\rho_{\delta_0}(t) \,\varrho_{\delta}(x-y)\, \psi(t,x) \varsigma_l(k)\,dy \,dx\,dt\,dk\Big]
\\ \to_{\delta_0 \to 0} & 
\E \Big[ \int\limits_{\R^d\times \R^d \times\R} \beta \big(u_{\Delta x}(0,y)-u_0^{\eps}(x)+k\big)
 \,\varrho_{\delta}(x-y)\, \psi(0,x) \varsigma_l(k)\,dy \,dx\,dk\Big]
 \\ \to_{\beta \to |\cdot|} & 
\E \Big[ \int\limits_{\R^d\times \R^d \times\R} \big|u_{\Delta x}(0,y)-u_0^{\eps}(x)+k\big|
 \,\varrho_{\delta}(x-y)\, \psi(0,x) \varsigma_l(k)\,dy \,dx\,dk\Big]
  \\ \to_{l \to +\infty} & 
\E \Big[ \int\limits_{\R^d\times \R^d} \big|u_{\Delta x}(0,y)-u_0^{\eps}(x)\big|
 \,\varrho_{\delta}(x-y)\, \psi(0,x) \,dy \,dx\Big].
\end{align*}
\guy{Note that the compact support of $\psi$ justifies the integrability of the integrands.}
\end{comment}
%%%%%%%%%%%%%%%%%

\begin{lemma}\cite{BaVaWit}\label{initial_terms} It holds that
\begin{align*}
\lim_{l \to 0}\lim_{\eta \to |\cdot|}\lim_{\gamma \to 0}\lim_{\delta_0\to 0}( I_1+J_1) &=
\E \left[ \int_{\R^d\times \R^d} \big|u_{\Dy}(0,y)-u_0^{\eps}(x)\big|
 \,\varrho_{\delta}(x-y)\, \psi(0,x) \,dy \,dx\right]. \\
 \lim_{l \to 0}\lim_{\eta \to |\cdot|}\lim_{\gamma \to 0}\lim_{\delta_0\to 0}( I_2+J_2) &=
 \E \bigg[\int_{\R^d\times\Pi_T} \big|u_\eps(t,x)-u_\Dy(t,y)\big| \,\varrho_{\delta}(x-y)\, \partial_t \psi(t,x) \, dx\,dt\,dy \bigg]. 
\end{align*}

\end{lemma}
\begin{lemma} \cite[Lemma 3.2]{BhKoleyVa1}\label{stochastic}
It holds that 
\begin{align*}
&\lim_{\eta \to |\cdot|}\lim_{\gamma \to 0}\lim_{\delta_0\to 0} (I_3+J_3)\\
& = -2E \left[  \int_{\Pi_T\times\R^d} \sum_{j \ge 1} g_j(u_{\Dy}(t,y))g_j(u_\eps(t,x))\psi(t,x) \,\rho_\delta (x-y) \rho_l(u_{\Dy}(t,y)-u_\eps(t,x))\,dy\,dx\,dt\right].\\
&\lim_{\eta \to |\cdot|}\lim_{\gamma \to 0}\lim_{\delta_0\to 0} (I_4+J_4)\\
&=\E \left[ \int_{\Pi_T\times\Pi_T} \Big[\mathbb{G}^2(u_\eps(t,x))+\mathbb{G}^2(u_\Dy(t,y))\Big]\, \psi(t,x) \rho_\delta(x-y) \, \rho_l(u_\eps(t,x)-u_{\Dy}(t,y))\,dy\,dx\,dt\,dk \right].
\end{align*}
Finally, it follows that
$$\lim_{l \to 0} \lim_{\eta \to |\cdot|}\lim_{\gamma \to 0}\lim_{\delta_0\to 0} (I_3+J_3+I_4+J_4)=0.$$ 
\end{lemma}

\noindent Next, we move on to estimate the terms coming from the associated flux function. 
\begin{lemma}
\label{flux} It holds that
\begin{align}
&\limsup_{l \rightarrow 0}\,\lim_{\eta \to |\cdot|}\,\lim_{\gamma \to 0}\lim_{\delta_0 \rightarrow 0} \, (I_5 +J_5) \\ & \qquad\le 
 \E\Big[\int_{\Pi_T}\int_{\R^d} Q \big(u_\eps(t,x),u_{\Dy}(t,y)\big)\cdot\nabla_x \psi(t,x) \varrho_{\delta}(x-y)\,dy\,dx\,dt\Big]
+ \frac{\Dy }{\delta}C \|u_0\|_{BV}\int_0^T \|\psi(t)\|_{\infty} dt.  \notag
\end{align}
\end{lemma}

\begin{proof}
We start with 
\begin{align*}
J_5=& - \E \bigg[\int_{\R \times\Pi_T\times\Pi_T}\hspace{-1cm}  \eta^{\prime}(u_{\Dy}(s,y)-k) \frac{1}{\Dy} \sum^d_{i=1}\Big[F_i(u_{\Dy}(s,y), u_{\Dy}(s,y+\Dy e_i))-F_i(u_{\Dy}(s,y-\Dy e_i), u_{\Dy}(s,y))\Big]
 \notag
  \\[-0.4cm] &\hspace{7 cm} \times
  \phi_{\delta,\delta_0}(t,x,s,y)  \varsigma_l(u^\gamma_{\eps}(t,x)-k)\,ds\,dy\,dx\,dt\,dk\bigg]
  \\ \overset{\delta_0 \to 0}\to &
  - \E \bigg[\int_{\R \times\Pi_T\times\R^d}\hspace{-1cm}  \eta^{\prime}(u_{\Dy}(t,y)-k) \frac{1}{\Dy} \sum^d_{i=1}\Big[F_i(u_{\Dy}(t,y), u_{\Dy}(t,y+\Dy e_i))-F_i(u_{\Dy}(t,y-\Dy e_i), u_{\Dy}(t,y))\Big]
 \notag
  \\[-0.4cm] &\hspace{7 cm} \times
  \rho_\delta(x-y)\psi(t,x)  \varsigma_l(u^\gamma_{\eps}(t,x)-k)\,ds\,dy\,dx\,dt\,dk\bigg]
  \\ \overset{\gamma \to 0}\to &
  - \E \bigg[\int_{\R \times\Pi_T\times\R^d}\hspace{-1cm}  \eta^{\prime}(u_{\Dy}(t,y)-k) \frac{1}{\Dy} \sum^d_{i=1}\Big[F_i(u_{\Dy}(t,y), u_{\Dy}(t,y+\Dy e_i))-F_i(u_{\Dy}(t,y-\Dy e_i), u_{\Dy}(t,y))\Big]
 \notag
  \\[-0.4cm] &\hspace{7 cm} \times
  \rho_\delta(x-y)\psi(t,x)  \varsigma_l(u_{\eps}(t,x)-k)\,ds\,dy\,dx\,dt\,dk\bigg]
  \\ \overset{\eta \to |.|}\to &
  - \E \bigg[\int_{\R \times\Pi_T\times\R^d}\hspace{-1cm}  \sign_0(u_{\Dy}(t,y)-k) \frac{1}{\Dy} \sum^d_{i=1}\Big[F_i(u_{\Dy}(t,y), u_{\Dy}(t,y+\Dy e_i))-F_i(u_{\Dy}(t,y-\Dy e_i), u_{\Dy}(t,y))\Big]
 \notag
  \\[-0.4cm] &\hspace{7 cm} \times
  \rho_\delta(x-y)\psi(t,x)  \varsigma_l(u_{\eps}(t,x)-k)\,ds\,dy\,dx\,dt\,dk\bigg].
\end{align*}
By arguments similar to the proof of Lemma \ref{lem:cellentropyinequality} and since, for any $v,w$,  
\begin{align*}
F_i(k, k \top v)-F_i(k \top w, k) - \Big[ F_i(k, k \bot v)-F_i(k \bot w, k) \Big] \leq 0,
\end{align*}
one gets that 
\begin{align*}
&- \sign_0(u_{\Dy}(t,y)-k) \frac{1}{\Dy} \Big[F_i(u_{\Dy}(t,y), u_{\Dy}(t,y+\Dy e_i))-F_i(u_{\Dy}(t,y-\Dy e_i), u_{\Dy}(t,y))\Big]
\\ \leq &
\Big[F_i(k \bot u_{\Dy}(t,y), k \bot u_{\Dy}(t,y+\Dy e_i))-F_i(k \bot u_{\Dy}(t,y-\Dy e_i), k \bot u_{\Dy}(t,y))\Big] 
\\ & - \Big[F_i(k \top u_{\Dy}(t,y), k \top u_{\Dy}(t,y+\Dy e_i))-F_i(k \top u_{\Dy}(t,y-\Dy e_i), k \top u_{\Dy}(t,y))\Big].
\end{align*}
Thus, 
\begin{align*}
&- \int_{\R^d} \sign_0(u_{\Dy}(t,y)-k)  \Big[F_i(u_{\Dy}(t,y), u_{\Dy}(t,y+\Dy e_i))-F_i(u_{\Dy}(t,y-\Dy e_i), u_{\Dy}(t,y))\Big]  \rho_\delta(x-y)\,dy
\\ &\leq 
\int_{\R^d}  \Big[F_i(k \bot u_{\Dy}(t,y), k \bot u_{\Dy}(t,y+\Dy e_i))-F_i(k \bot u_{\Dy}(t,y-\Dy e_i), k \bot u_{\Dy}(t,y))\Big]  \rho_\delta(x-y)\,dy
\\ &
- \int_{\R^d} \Big[F_i(k \top u_{\Dy}(t,y), k \top u_{\Dy}(t,y+\Dy e_i))-F_i(k \top u_{\Dy}(t,y-\Dy e_i), k \top u_{\Dy}(t,y))\Big]  \rho_\delta(x-y)\,dy.
\\&=\int_{\R^d} 
\Big[f_i(k \bot u_{\Dy}(t,y))-f_i(k \bot u_{\Dy}(t,y-\Dy e_i))\Big]  \rho_\delta(x-y)\,dy
\\ &
- \int_{\R^d} 
\Big[f_i(k \top u_{\Dy}(t,y))-f_i(k \top u_{\Dy}(t,y-\Dy e_i))\Big]  \rho_\delta(x-y)\,dy
\\  &+
\sum_{\alpha \in \Z^d} [F_i(k \bot u_\alpha(t), k \bot u_{\alpha + e_i}(t))-f_i(k \bot u_{\alpha}(t))]  \Big[\int_{R_\alpha}\rho_\delta(x-y)\,dy-\int_{R_{\alpha+e_i}}\rho_\delta(x-y)\,dy\Big]  
\\ &
- \sum_{\alpha \in \Z^d} [F_i(k \top u_\alpha(t), k \top u_{\alpha + e_i}(t)) -f_i(k \top u_{\alpha}(t))] \Big[\int_{R_\alpha}\rho_\delta(x-y)\,dy-\int_{R_{\alpha+e_i}}\rho_\delta(x-y)\,dy\Big]  
\\ \leq &\int_{\R^d} 
\Big[f_i(k \bot u_{\Dy}(t,y))-f_i(k \bot u_{\Dy}(t,y-\Dy e_i))\Big]  \rho_\delta(x-y)\,dy
\\ &
- \int_{\R^d} 
\Big[f_i(k \top u_{\Dy}(t,y))-f_i(k \top u_{\Dy}(t,y-\Dy e_i))\Big]  \rho_\delta(x-y)\,dy
\\  &+
\sum_{\alpha \in \Z^d} C(F_i)|u_{\alpha + e_i}(t)- u_{\alpha}(t)|  \int_{R_\alpha}|\rho_\delta(x-y)-\rho_\delta(x-y-\Dy e_i)|\,dy.
\end{align*}
Since 
\begin{align*}
&\E \bigg[\int_{\R \times\Pi_T}  \sum_{\alpha \in \Z^d} |u_{\alpha + e_i}(t)- u_{\alpha}(t)|  \int_{R_\alpha}|\rho_\delta(x-y)-\rho_\delta(x-y-\Dy e_i)|\,dy
 \notag
  \\[-0.5cm] &\hspace{9 cm} \times
\frac1\Dy \psi(t,x)  \varsigma_l(u_{\eps}(t,x)-k)\,ds\,dy\,dx\,dt\,dk\bigg]
\\=
&\frac1\Dy\E \bigg[\int_{\Pi_T}  \sum_{\alpha \in \Z^d} |u_{\alpha + e_i}(t)- u_{\alpha}(t)|  \int_{R_\alpha}|\rho_\delta(x-y)-\rho_\delta(x-y-\Dy e_i)|\,dy  \psi(t,x)  \,dx\,dt\bigg]
\\ \leq
&\frac1\Dy\bigg[\int_0^T \|\psi(t)\|_{\infty} \sum_{\alpha \in \Z^d} \E|u_{\alpha + e_i}(t)- u_{\alpha}(t)|  \int_{\R^d} \int_{R_\alpha}|\rho_\delta(x-y)-\rho_\delta(x-y-\Dy e_i)|\,dy\,dx\,dt\bigg]
\\ \leq
&\frac1\Dy\bigg[\int_0^T \|\psi(t)\|_{\infty} \sum_{\alpha \in \Z^d} \E|u_{\alpha + e_i}(t)- u_{\alpha}(t)|  C\frac{\Dy}{\delta}\Dy^d dt\bigg] \leq C \|u_0\|_{BV}\int_0^T \|\psi(t)\|_{\infty} dt \frac{\Dy}{\delta},
\end{align*}
we are left to 
\begin{align*}
&\lim_{\delta_0,\gamma,\eta}J_5 \leq C \|u_0\|_{BV}\int_0^T \|\psi(t)\|_{\infty} dt \frac{\Dy}{\delta}
\\
& + \frac{1}{\Dy}\E \bigg[\int\limits_{\R \times\Pi_T}\hspace{-0.25cm}   \sum^d_{i=1}\Big[\int_{\R^d}  \Big[f_i(k \bot u_{\Dy}(t,y))-f_i(k \bot u_{\Dy}(t,y-\Dy e_i))\Big]  \rho_\delta(x-y)\,dy\Big]
 \notag
  \\[-0.5cm] &\hspace{9 cm} \times
  \psi(t,x)  \varsigma_l(u_{\eps}(t,x)-k)\,dx\,dt\,dk\bigg]
\\
& -\frac{1}{\Dy}\E \bigg[\int\limits_{\R \times\Pi_T}\hspace{-0.25cm}   \sum^d_{i=1}\Big[\int_{\R^d}  \Big[f_i(k \top u_{\Dy}(t,y))-f_i(k \top u_{\Dy}(t,y-\Dy e_i))\Big]  \rho_\delta(x-y)\,dy\Big]
 \notag
  \\[-0.5cm] &\hspace{9 cm} \times
  \psi(t,x)  \varsigma_l(u_{\eps}(t,x)-k)\,dx\,dt\,dk\bigg]\\
  &=\frac{C  \Dy \|u_0\|_{BV}}{\delta}\int_0^T \|\psi(t)\|_{\infty} dt
  \\ & - \frac{1}{\Dy}\E \bigg[\int\limits_{\R \times\Pi_T}\hspace{-0.25cm}   \sum^d_{i=1}\Big[\int_{\R^d}  
  \Big[\{Q(u_{\Dy}(t,y),k)\}_i - \{Q(u_{\Dy}(t,y-\Dy e_i),k)\}_i\Big]  \rho_\delta(x-y)\,dy\Big]
   \notag
    \\[-0.5cm] &\hspace{9 cm} \times
    \psi(t,x)  \varsigma_l(u_{\eps}(t,x)-k)\,dx\,dt\,dk\bigg].
\end{align*}
%Set $\{F(u,v)\}_i$ the component $i$ of $F(u,v)$ and note that $\{F(u,v)\}_i=\sign(u-v)[f_i(u)-f_i(v)]=f_i(u\top v)-f_i(u\bot v )$ and 
%\begin{align*}
%&\lim_{\delta_0,\gamma,\eta}J_5 \leq \frac{\Dy C \|u_0\|_{BV}}{\delta}\int_0^T \|\psi(t)\|_{\infty} dt
%\\
%& - \frac{1}{\Dy}\E \bigg[\int\limits_{\R \times\Pi_T}\hspace{-0.25cm}   \sum^d_{i=1}\Big[\int_{\R^d}  
%\Big[\{F(u_{\Dy}(t,y),k)\}_i - \{F(u_{\Dy}(t,y-\Dy e_i),k)\}_i\Big]  \rho_\delta(x-y)\,dy\Big]
% \notag
%  \\[-0.5cm] &\hspace{9 cm} \times
%  \psi(t,x)  \varsigma_l(u_{\eps}(t,x)-k)\,dx\,dt\,dk\bigg].
%\end{align*}
Here $\{Q(u,v)\}_i$ denotes the $i$-th component of $Q(u,v)$. Next, passing to the limit over $l$ yields 
\begin{align*}
&\limsup_l\lim_{\delta_0,\gamma,\eta}J_5 \leq \frac{\Dy C \|u_0\|_{BV}}{\delta}\int_0^T \|\psi(t)\|_{\infty} dt
\\
& - \frac{1}{\Dy}\E \bigg[\int\limits_{\Pi_T}  \sum^d_{i=1}\int_{\R^d}  
\{Q(u_{\Dy}(t,y),u_{\eps}(t,x))\}_i - \{Q(u_{\Dy}(t,y-\Dy e_i),u_{\eps}(t,x))\}_i  \rho_\delta(x-y)\,dy  \psi(t,x) \,dx\,dt\bigg].
\end{align*}
Note that,
\begin{align*}
&-\frac1\Dy \int_{\R^d}  \{Q(u_{\Dy}(t,y),u_{\eps}(t,x))\}_i - \{Q(u_{\Dy}(t,y-\Dy e_i),u_{\eps}(t,x))\}_i  \rho_\delta(x-y)\,dy
\\ =&
-\frac1\Dy \int_{\R^d}  \{Q(u_{\Dy}(t,y),u_{\eps}(t,x))\}_i \Big[ \rho_\delta(x-y)-\rho_\delta(x-y-\Dy e_i)\Big]\,dy
\\ =&
-\int_{\R^d}  \{Q(u_{\Dy}(t,y),u_{\eps}(t,x))\}_i \Big[ \frac{\rho_\delta(x-y)-\rho_\delta(x-y-\Dy e_i)}{\Dy}-\partial_i\rho_\delta(x-y)+\partial_i\rho_\delta(x-y)\Big]\,dy.
\end{align*}
Then, 
\begin{align*}
&\Big|\int_{\R^d}  \{Q(u_{\Dy}(t,y),u_{\eps}(t,x))\}_i \Big[ \frac{\rho_\delta(x-y)-\rho_\delta(x-y-\Dy e_i)}{\Dy}-\partial_i\rho_\delta(x-y)\Big]\,dy\Big|
\\ = &
\Big|\int_{\R^d}  \{Q(u_{\Dy}(t,y),u_{\eps}(t,x))\}_i  \int_{-\Dy}^0\Big[\partial_i\rho_\delta(x-y+\tau e_i)-\partial_i\rho_\delta(x-y)\Big]d\tau\,dy\Big|
\\ = &
\Big|\int_{\R^d} \int_{-\Dy}^0\Big[\{Q(u_{\Dy}(t,y+\tau e_i),u_{\eps}(t,x))\}_i -Q(u_{\Dy}(t,y),u_{\eps}(t,x))\}_i\Big]\partial_i\rho_\delta(x-y)d\tau\,dy\Big|
\\ \leq &
\|f\|_{Lip} \int_{\R^d} \int_{-\Dy}^0\Big|u_{\Dy}(t,y+\tau e_i) - u_{\Dy}(t,y)\Big| |\partial_i\rho_\delta(x-y)| d\tau\,dy.
\end{align*}
So that
\begin{align*}
&\frac{\|f\|_{Lip}}{\Dy}\E \bigg[\int\limits_{\Pi_T}   \sum^d_{i=1}\int_{\R^d}  
\int_{-\Dy}^0\Big|u_{\Dy}(t,y+\tau e_i) - u_{\Dy}(t,y)\Big| |\partial_i\rho_\delta(x-y)| d\tau\,dy  \psi(t,x)  \,dx\,dt\bigg]
\\ \leq &
\frac{\|f\|_{Lip}}{\Dy} \sum^d_{i=1}\E \bigg[\int_0^T \|\psi(t)\|_\infty   \int_{\R^d}  
\int_{-\Dy}^0\Big|u_{\Dy}(t,y+\tau e_i) - u_{\Dy}(t,y)\Big| \,d\tau\int\limits_{\R^d} |\partial_i\rho_\delta(x-y)| \,dx \,dy    \,dt\bigg]
\\ \leq &
\frac{\|f\|_{Lip}}{\delta\Dy} \sum^d_{i=1}\E \bigg[\int_0^T \|\psi(t)\|_\infty  \int_{-\Dy}^0 \int_{\R^d}  
\Big|u_{\Dy}(t,y+\tau e_i) - u_{\Dy}(t,y)\Big|\,dy \,d\tau    \,dt\bigg]
\\ \leq &
\frac{\|f\|_{Lip}}{\delta\Dy} \sum^d_{i=1} \bigg[\int_0^T \|\psi(t)\|_\infty  \int_{-\Dy}^0 \Dy^d \sum_\alpha  \E
\Big|u_{\alpha-e_i}(t) - u_{\alpha}(t)\Big|\,dy \,d\tau    \,dt\bigg]
\\ \leq &
\frac{\|f\|_{Lip}\Dy}{\delta} \|u_0\|_{BV}\int_0^T \|\psi(t)\|_\infty\,dt.
\end{align*}
In conclusion, 
\begin{align*}
&\limsup_l\lim_{\delta_0,\gamma,\eta}J_5 \\
&\leq \frac{\Dy }{\delta}C \|u_0\|_{BV}\int_0^T \|\psi(t)\|_{\infty} dt
%\\& 
- \frac{1}{\Dy}\E \bigg[\int_{\Pi_T\times \R^d}  
Q(u_{\Dy}(t,y),u_{\eps}(t,x)).\nabla \rho_\delta(x-y)  \psi(t,x) \,dy\,dx\,dt\bigg].
\end{align*}

\noindent For the other term, we simply follow \cite{BaVaWit} to conclude,
\begin{align*}
&\limsup_l\lim_{\delta_0,\gamma,\eta} I_5 = \E\bigg[\int_{\Pi_T}\int_{\R^d}  Q(u_\eps(t,x),u_{\Dy}(t,y))\cdot\nabla_x (\varrho_{\delta}(x-y)\psi(t,x) )\,dy \,dx\,dt\bigg],
%I_5 &= \E \bigg[\int_{\Pi_T}\int_{\Pi_T} \int_{\R} 
% F^\eta(u^\gamma_\eps(t,x),k)\cdot\nabla_x \phi_{\delta,\delta_0}(t,x,s,y) \varsigma_l(u_{\Dy}(s,y)-k)\,dk\,dx\,dt\,dy\,ds\bigg]\\
% &\overset{\delta_0 \to 0}\longrightarrow \E\bigg[\int_{\Pi_T}\int_{\R^d}\int_{\R} F^\eta \big(u^\gamma_\eps(t,x),k \big)\cdot\nabla_x( \varrho_{\delta}(x-y)\psi(t,x))\varsigma_l(u_{\Dy}(t,y)-k)\,dk \,dy \,dx\,dt\bigg]  \\
%  &\overset{\gamma \to 0}\longrightarrow \E\bigg[\int_{\Pi_T}\int_{\R^d}\int_{\R} F^\eta \big(u_\eps(t,x),k \big)\cdot\nabla_x( \varrho_{\delta}(x-y)\psi(t,x))\varsigma_l(u_{\Dy}(t,y)-k)\,dk \,dy \,dx\,dt\bigg]  \\
% &\overset{\eta \to |.|} \longrightarrow \E\bigg[\int_{\Pi_T}\int_{\R^d}\int_{\R}  F (u_\eps(t,x),k )\cdot \nabla_x (\varrho_{\delta}(x-y)\psi(t,x))\varsigma_l(u_{\Dy}(t,y)-k)\,dk \,dy \,dx\,dt\bigg] \\
% &\overset{l \to 0} \longrightarrow \E\bigg[\int_{\Pi_T}\int_{\R^d}  F(u_\eps(t,x),u_{\Dy}(t,y))\cdot\nabla_x (\varrho_{\delta}(x-y)\psi(t,x) )\,dy \,dx\,dt\bigg]
\end{align*}
and the lemma is proved.
\end{proof}

\begin{lemma}\cite{BaVaWit}\label{viscous} It holds that
\begin{align*}
\lim_{\eps \to 0}\limsup_{l \to 0}\limsup_{\eta \to |\cdot|}\limsup_{\delta_0\to 0} |I_8| = 0 . 
\end{align*}
\end{lemma}

%%%%%%%%%%%%%%%%%%%%%%%%%%%%%%%%%%%%%%%%%%%%%%%%%%%%%%%%
We now estimate the terms coming from the fractional operator in the next two lemmas. We choose $r=\Dy$ for the subsequent calculations. 
\begin{lemma}\label{fractional_upper} It holds that
\begin{align*}
\limsup_{l \to 0}\limsup_{\eta \to |\cdot|}\limsup_{\delta_0\to 0} (I_6+J_6) 
& \leq 
\E \bigg[\int\limits_{\R^d\times\Pi_T}\hspace{-0.25cm}  |A(u_\Dy(t,y))-A(u_\eps(t,x))| \mathfrak{L}_\lambda^r[ \psi(t,\cdot)](x)\rho_\delta(x-y) \,dy\,dx\,dt\bigg]\\ &\hspace{3cm}+ \frac{C_\lambda}{\delta}\|A'\|_\infty\|\psi\|_\infty \|u_0\|_{BV} 
\begin{cases}
\Dy,  & \text{if} \,\,  \lambda<1/2,  \\
\Dy |\ln \Dy|,  & \text{if}  \,\, \lambda = 1/2, \\
\Dy^{2(1-\lambda)}, & \text{if}   \,\,\lambda >1/2.
\end{cases}
\end{align*}
\end{lemma}
\begin{proof}
We have
\begin{align*}
I_6+J_6=& - \E \left[\int_{\R \times\Pi_T\times\Pi_T} \mathfrak{L}_{\lambda}^{\Dy}[A(u_\Dy (s,\cdot))](y)\, \overline \phi^y_{\delta,\delta_0}(t,x,s,y)\, \eta'(u_{\Dy}(s,y) -k) \,dy\,ds 
\, \varsigma_l(u_\eps(t,x)-k)\,dx\,dt\,dk\right]\\
&\quad - \E \left[\int_{\R \times \Pi_T\times\Pi_T} \mathfrak{L}_{\lambda}^r[A(u_\eps(t,\cdot))](x)\, \phi_{\delta,\delta_0}(t,x,s,y)\, \eta'(u_\eps(t,x) -k) \,dx\,dt 
\, \varsigma_l(u_{\Dy}(s,y)-k)\,dy\,ds\,dk \right]\\
=&-\E \left[\int_{\R \times\Pi_T\times\Pi_T} \mathfrak{L}_{\lambda}^r[A(u_{\Dy}(s,\cdot))](y)\, \phi_{\delta,\delta_0}(t,x,s,y)\, \eta'(u_{\Dy}(s,y) -k) \,dy\,ds 
\, \varsigma_l(u_{\eps}(t,x)-k)\,dx\,dt\,dk \right]\\
&\quad - \E \left[\int_{\R\times \Pi_T\times\Pi_T} \mathfrak{L}_{\lambda}^r[A(u_\eps(t,\cdot))](x)\, \phi_{\delta,\delta_0}(t,x,s,y)\, \eta'(u_\eps(t,x) -k) \,dx\,dt 
\, \varsigma_l(u_{\Dy}(s,y)-k)\,dy\,ds\,dk \right]\\
&\quad -\E \bigg[\int_{\R \times \Pi_T\times\Pi_T} \mathfrak{L}_{\lambda}^{r}[A(u_{\Dy}(s,\cdot))](y)\, \Big(\overline \phi^y_{\delta,\delta_0}(t,x,s,y)-\phi_{\delta,\delta_0}(t,x,s,y)\Big)\, \eta'(u_{\Dy}(s,y) -k) \,dy\,ds\\
&\hspace{7cm}\times \varsigma_l(u_\eps(t,x)-k)\,dx\,dt\,dk\bigg]\\
&:=A+B+C.
\end{align*}

%We first consider ,
%\begin{align*}
%A+B=&- \E \left[\int_{\R \times \Pi_T\times\Pi_T} \mathfrak{L}_{\lambda}^{r}[A(u_{\Dy}(s,\cdot))](y)\, \phi_{\delta,\delta_0}(t,x,s,y)\, \eta'(u_{\Dy}(s,y) -k) \,dy\,ds 
%\, \varsigma_l(u_\eps(t,x)-k)\,dx\,dt\,dk\right]\\
%&\quad - \E \left[\int_{\R \times \Pi_T\times\Pi_T} \mathfrak{L}_{\lambda}^r[A(u_\eps(t,\cdot))](x)\, \phi_{\delta,\delta_0}(t,x,s,y)\, \eta'(u_\eps(t,x) -k) \,dx\,dt 
%\, \varsigma_l(u_{\Dy}(s,y)-k)\,dy\,ds\,dk \right]\\
%\end{align*}

\noindent Following \cite[Lemma 3.4, Lemma 4.6]{BhKoleyVa1} we have,
\begin{align*} 
&\lim_{\eta \to |\cdot|} \lim_{\delta_0 \to 0}( A+B)\\
&=- \E \left[\int_{\R \times \R^d \times\Pi_T} \mathfrak{L}_{\lambda}^{r}[A(u_{\Dy}(t,\cdot))](y)\,\psi(t,x) \rho_\delta(x-y) \sign(u_{\Dy}(t,y) -k)
\, \varsigma_l(u_\eps(t,x)-k)\,dx\,dt\,dy\,dk\right]\\
&\quad - \E \left[\int_{\R \times \R^d\times\Pi_T} \mathfrak{L}_{\lambda}^r[A(u_\eps(t,\cdot))](x)\, \psi(t,x) \rho_\delta(x-y)\, \sign(u_\eps(t,x) -k) 
\, \varsigma_l(u_{\Dy}(t,y)-k)\,dx\,dt\,dx\,dk \right]\\
&=- \E \bigg[\int_{\R \times \R^d\times\Pi_T} \left(\mathfrak{L}_{\lambda}^{r}[A(u_{\Dy}(s,\cdot))](y)- \mathfrak{L}_{\lambda}^r[A(u_\eps(t,\cdot))](x)\right)\sgn(u_{\Dy}(s,y)-u_\eps(t,x)+k)\\
&\hspace{6cm}\times\psi(t,x)\rho_{\delta_0}(t-s) \rho_\delta(x-y) \varsigma_l(k)\,dx\,dt\,dy\bigg]\\
&\le -\E \bigg[\int_{\R^d\times\Pi_T} | A(u_\Dy(t,y))-A(u_\eps(t,x))| \mathfrak{L}^r_\lambda[ \psi(t,\cdot)](x) \rho_{\delta}(x-y) \,dx\,dt\,dy\bigg]+ C(\lambda)\frac{\|A'\|_\infty l}{r^{2\lambda}}.
\end{align*}
Thus we conclude that 
\begin{align*}
\limsup_{ l \to 0} \lim_{\eta \to |\cdot|} \lim_{\delta_0 \to 0}
(A+B) \le -\E \bigg[\int_{\R^d\times\Pi_T} | A(u_\Dy(t,y))-A(u_\eps(t,x))| \mathfrak{L}^r_\lambda[ \psi(t,\cdot)](x) \rho_{\delta}(x-y) \,dx\,dt\,dy\bigg].
\end{align*}
We now consider,
\begin{align*}
C=&-\E \bigg[\int_{\R \times \Pi_T\times\Pi_T} \mathfrak{L}_{\lambda}^{r}[A(u_{\Dy}(s,\cdot))](y)\, \Big(\overline \phi^y_{\delta,\delta_0}(t,x,s,y)-\phi_{\delta,\delta_0}(t,x,s,y)\Big)\, \eta'(u_{\Dy}(s,y) -k) \,dy\,ds\\
&\hspace{6cm}\times \varsigma_l(u_\eps(t,x)-k)\,dx\,dt\,dk\bigg]\\
 &= \E \bigg[\int_{\Pi_T\times\Pi_T\times\R} \int_{|z|>r} [A(u_{\Dy}(s,y))-A(u_{\Dy}(s,y+z))]d\mu_\lambda(z)\, \psi(t,x)\rho_{\delta_0}(t-s)\big(\overline \rho_\delta^y(x,y)-\rho_{\delta}(x-y)\big)\,  \\
 &\hspace{6cm} \times\eta'(u_{\Dy}(s,y) -k)\varsigma_l(u_\eps(t,x)-k)\,dk\,dy\,ds\,dx\,dt\bigg].
\end{align*}
Since $|\eta'| \leq 1$, 
\begin{align*}
|C|
\leq & \E \Big[\int\limits_{\Pi_T\times\Pi_T} \int_{|z|>r} \big|A(u_\Dy(s,y))-A(u_\Dy(s,y+z))\big|d\mu_\lambda(z)\, \psi(t,x)\rho_{\delta_0}(t-s)\big|\overline \rho_\delta^y(x,y)-\rho_{\delta}(x-y)\big|\,  \,dy\,ds \,dx\,dt\Big]
\\
\leq &\|A'\|_\infty\|\psi\|_\infty \E \Big[\int\limits_{\Pi_T} \int_{|z|>r} \big|u_{\Dy}(s,y)-u_\Dy(s,y+z)\big|d\mu_\lambda(z)\, \int\limits_{\R^d}\big|\overline \rho_\delta^y(x,y)-\rho_{\delta}(x-y)\big|\, dx \,dy\,ds \Big]
\\ \leq & C\frac{\Dy}{\delta}\|A'\|_\infty\|\psi\|_\infty\left[\|u_0\|_{BV} \int_{1 \geq |z|>r} |z| d\mu_{\lambda}(z) +\|u_0\|_{L^1} \int_{|z|>1} d\mu_{\lambda}(z) \right],
\end{align*}
where to derive the last inequality we have used the fact that $\int_{\R^d}| \overline{\rho}_\eps(y)-\rho_\eps(y)|\,dy \le \frac{\Dy}{\eps}$ for any space mollifier $\rho_\eps$. Now we conclude the proof of the lemma with the aid of the following:
\begin{align*}
\int_{1 \geq |z|>\Dy} |z| d\mu_{\lambda}(z) \leq 
\begin{cases}
C_\lambda,  & \text{if}\,\,   \lambda<1/2,  \\
C_\lambda |\ln \Dy|,  & \text{if}\,\,   \lambda = 1/2, \\
C_\lambda  \Dy^{1-2\lambda}, & \text{if}\,\,   \lambda >1/2 .
\end{cases}
\end{align*}
\end{proof}

\begin{lemma}\label{fractional_lower} It holds that
\begin{align*}
&\limsup_{l \to 0}\limsup_{\eta \to |\cdot|}\limsup_{\delta_0\to 0}(I_7+J_7)\le \\
&\qquad \|u_0\|_{BV} \|A'\|_\infty\int^T_0 \bigg( \| \nabla\psi(t,\cdot)\|_\infty+ \frac{C}{\delta}\| \psi(t,\cdot)\|_\infty\bigg)\int_{|z|\le r} |z|^2 \,d\mu_\lambda(z)\,dt + C \frac{r^{2-2\lambda}}{\delta}.
\end{align*}
\end{lemma}

\begin{proof}
Following \cite[Lemma 3.5]{BhKoleyVa1} we have,
\begin{align*}
&\limsup_{\delta \to 0} \lim_{\gamma \to 0} \lim_{n \to \infty} I_7\\
&\qquad \le -E \left[\int_{\Pi_T \times \R^d \times \R}|A(u_\eps(t,x))-A(k)| \mathfrak{L}_{\lambda,r}[\psi(t,\cdot)\rho_\delta(\cdot -y)](x) \varsigma_l(u_\Dy(t,y)-k)\,dk\,dy\,dx\,dt\right] \\
&=\E\bigg[ \int_{\Pi_T\times\R^d\times\R} |A(u_\eps(t,x)) -A(k)| \,\int_{|z|\leq r}\int^1_0(\tau -1) z^T.D^2(\psi(t,x)\rho_\delta(x-y+\tau z)).z \,d\tau\,d\mu_\lambda(z)\\[-0.1cm]
&\hspace{8cm}\times\, \varsigma_l(u_\Dy(s,y)-k)\,dk\,dx \,dt\,dy\,ds \bigg]\\
&\le \E\bigg[ \int_{\Pi_T\times\R^d\times\R}\int_{|z|\leq r}\int^1_0 | \nabla A(u_\eps(t,x))| | \nabla( \psi(t,x) \rho_\delta(x-y+\tau z))|\, |z|^2\,d\tau\,d\mu_\lambda(z)\\[-0.2cm]
&\hspace{8cm}\times\,  \varsigma_l(u_\Dy(s,y)-k)\,dk\,dx \,dt\,dy\,ds \bigg]\\
&\le \E\bigg[ \int_{\Pi_T}\int_{|z|\leq r} | \nabla A(u_\eps(t,x))||\nabla \psi(t,x)| |z|^2\,d\mu_\lambda(z)\,dx \,dt \bigg]\\
&\quad+\E\bigg[ \int_{\Pi_T\times\R^d}\int_{|z|\leq r}\int^1_0 |\nabla A(u_\eps(t,x))| \psi(t,x)| \nabla \rho_\delta(x-y+\tau z)||z|^2\,d\tau\,d\mu_\lambda(z)\,dx \,dt\,dy \bigg]\\
&\le \|u_0\|_{BV}\|A'\|_\infty \int^T_0 \bigg( \| \nabla\psi(t,\cdot)\|_\infty+ \frac{C}{\delta}\| \psi(t,\cdot)\|_\infty\bigg)\int_{|z|\le r} |z|^2 \,d\mu_\lambda(z)\,dt.
\end{align*}
To handle the other term we consider,
\begin{align*}
I_7= &-  \E \bigg[\int_{\Pi_T\times\Pi_T\times\R} A^\eta_k(u_\Dy(s,y)) \,\mathfrak{L}_{\lambda,r}^{\frac{\Dy}{2}} [\overline \phi^y_{\delta,\delta_0}(t,x,s,\cdot)](y)
\, \varsigma_l(u_\eps(t,x)-k)\,dy\,ds\,dx \,dt\,dk \bigg]\\
=&-  \E \bigg[\int_{\Pi_T\times\Pi_T\times\R} A_k^\eta(u_\Dy(s,y)) \int_{\frac{\Dy}{2}<|z|\leq r} \Big(\overline \rho^y_{\delta}(x,y)-\overline \rho^y_{\delta}(x,y+z)\Big)d\mu_\lambda(z)\\
&\hspace{6cm}\times\psi(t,x)\rho_{\delta_0}(t-s) \varsigma_l(u_\eps(t,x)-k)\,dk\,dy\,ds\,dx \,dt \bigg].
\end{align*}

Following \cite{Cifani}, we denote in a first step
 \begin{align*}
I^\theta_7= &-  \E \bigg[\int\limits_{\Pi_T\times\Pi_T\times\R} A_k^\eta(u_\Dy(s,y)) \int_{\frac{\Dy}{2}<|z|\leq r} \Big((\overline \rho^y_{\delta})_\theta(x,y)-(\overline \rho^y_{\delta})_\theta(x,y+z)\Big)d\mu_\lambda(z)\\
&\hspace{6cm}\times \, \psi(t,x)\rho_{\delta_0}(t-s)\varsigma_l(u_\eps(t,x)-k) \,dk\,dy\,ds\,dx \,dt \bigg],
\end{align*}
where $(\overline \rho^y_{\delta})_\theta$ denotes a regularization of $\overline \rho^y_{\delta}$ in the $y$-variable by a mollification of parameter $\theta$. Then,
\begin{align*}
|I^\theta_7|= &  \Bigg|\E \bigg[\int_{\Pi_T\times\Pi_T\times\R} A_k^\eta(u_\Dy(s,y)) \int_{\frac{\Dy}{2}<|z|\leq r}\int^1_0 (1-\tau)z^T.D^2(\overline \rho^y_{\delta})_\theta(x,y+\tau z).z\,d\tau\,d\mu_\lambda(z)\\
&\hspace{6cm}\times \, \psi(t,x)\rho_{\delta_0}(t-s)\varsigma_l(u_\eps(t,x)-k) \,dk\,dy\,ds\,dx \,dt \bigg]\Bigg|\\
&= \Bigg|\E \bigg[\int_{\Pi_T\times\Pi_T\times\R} \int_{\frac{\Dy}{2}<|z|\leq r}\int^1_0 (1-\tau)\nabla(\overline \rho^y_{\delta})_\theta(x,y+\tau z).z\,d(\nabla A_k^\eta(u_\Dy(s,\cdot)))(y).z\,d\tau\,d\mu_\lambda(z)\\
&\hspace{6cm}\times \, \psi(t,x)\rho_{\delta_0}(t-s)\varsigma_l(u_\eps(t,x)-k) \,dk\,dy\,ds\,dx \,dt \bigg]\Bigg|\\
&\le \|A'\|_\infty \E \bigg[\int_{\Pi_T\times\Pi_T} \int_{\frac{\Dy}{2}<|z|\leq r}\int^1_0 |\nabla(\overline \rho^y_{\delta})_\theta(x,y+\tau z)|\,d(|\nabla u_\Dy(s,\cdot)|)(y)|z|^2\,d\tau\,d\mu_\lambda(z)\\
&\hspace{6cm}\times \, \psi(t,x)\rho_{\delta_0}(t-s)\,dy\,ds\,dx \,dt \bigg]\\
&\le \frac{C\|A'\|_\infty{\|u_0\|_{BV}\|\psi\|_\infty}}{\delta}\int_{\frac{\Dy}{2}<|z|\le r} |z|^2\,d\mu_\lambda(z)\le \frac{C}{\delta} \int_{|z| \le r} |z|^2\,d\mu_\lambda(z) = C\frac{r^{2-2\lambda}}{\delta},
\end{align*}
where to derive the penultimate inequality we follow \cite[Proof of Thm 7.1]{Cifani} and use the fact that $\|( \overline{\rho}_\delta^y)_\theta\|_{BV} \le \|( \overline{\rho}_\delta^y)\|_{BV} \le \frac{C}{\delta}$. Finally to get the same estimate for $I_7$, we use that $\lim_{\theta \to 0}I^\theta_7= I_7$.
\end{proof}

\noindent Thus making using of Lemmas \ref{initial_terms}- \ref{fractional_lower}, and recalling that $r=\Dy$, we pass to the limit in $\eps$  to get 
\begin{align}
0\le &\E \left[ \int_{\R^d}\int_{\R^d} \big|u_{\Dy}(0,y)-u_0(x)\big|
 \,\varrho_{\delta}(x-y)\, \psi(0,x) \,dy \,dx\right] \label{collect}\\
 &+ \E \left[\int_{\Pi_T}\int_{\R^d} \big|u(t,x)-u_\Dy(t,y)\big| \,\varrho_{\delta}(x-y)\, \partial_t \psi(t,x) \,dy\, dx\,dt \right]\nonumber\\
 &+\E\left[\int_{\Pi_T}\int_{\R^d} Q \big(u(t,x),u_{\Dy}(t,y)\big)\cdot \nabla \psi(t,x)\varrho_{\delta}(x-y)\,dy\,dx\,dt\right]
+ C\frac{\Dy}{\delta}\nonumber\\
&+\E \left[\int_{\Pi_T}\int_{\R^d}  |A(u_\Dy(t,y))-A(u(t,x))| \mathfrak{L}_\lambda^{\Dy}[ \psi(t,\cdot)](x)\rho_\delta(x-y) \,dy\,dx\,dt\right] \nonumber\\
&+ \frac{C_\lambda}{\delta}\|\psi\|_\infty \|u_0\|_{BV} 
\begin{cases}
\Dy,  & \text{if} \,\,  \lambda<1/2  \\
\Dy |\ln \Dy|,  & \text{if}  \,\, \lambda = 1/2 \\
\Dy^{2(1-\lambda)}, & \text{if}   \,\,\lambda >1/2
\end{cases}\nonumber\\
&+\|u_0\|_{BV} \int^T_0 \bigg( \| \nabla\psi(t,\cdot)\|_\infty+ \frac{C}{\delta}\| \psi(t,\cdot)\|_\infty\bigg)\int_{|z|\le \Dy} |z|^2 \,d\mu_\lambda(z)\,dt+C\frac{(\Delta y)^{2-2\lambda}}{\delta}. \nonumber
\end{align}

To proceed further, we make a special choice for the function $\psi(t,x)$. To this end, for each $h>0$ and fixed $t\ge 0$, we define
 \begin{align}
 \psi_h^t(s)=\begin{cases} 1, &\quad \text{if}~ s\le t, \notag \\
 1-\frac{s-t}{h}, &\quad \text{if}~~t\le s\le t+h,\quad \psi_R(x)=\text{min}\Big(1,\frac{R^a}{|x|^a}\Big)\notag \\
 0, & \quad \text{if} ~ s \ge t+h.
 \end{cases}
 \end{align}
Furthermore, let $\rho$ be any non-negative mollifier.  Clearly, \eqref{collect} holds with $\psi(s,x)=\psi_h^t(s) \, (\psi_R \star \rho)(x)$. At this point, we can closely follow Bhauryal et. al. \cite{BhKoleyVa, BhKoleyVa1} and pass to the limits as $R \to \infty$ and $h \to 0$ to conclude 
\begin{align*}
\E \bigg[\int_{\R^d} \big|u(t,y)-u_\Dy(t,y)\big|  \,dx \bigg] \le& \E \Big[ \int_{\R^d} \big|u_{\Dy}(0,y)-u_0(y)\big|
 \,dx\Big]\\[-0.4cm]
 & +C\left(\delta+\frac{\Dy}{\delta}+\frac{(\Delta y)^{2-2\lambda}}{\delta}\right) + \frac{C}{\delta}  
\begin{cases}
\Dy,  & \text{if} \,\,  \lambda<1/2,  \\
\Dy |\ln \Dy|,  & \text{if}  \,\, \lambda = 1/2, \\
\Dy^{2(1-\lambda)}, & \text{if}   \,\,\lambda >1/2.
\end{cases}
\end{align*}
For $\lambda \in (0,1/2]$, we choose $\delta =\sqrt{\Dy}$ and for $\lambda \in (1/2,1)$, we choose $\delta = ( \Dy)^{1-\lambda}$ to obtain the following
\begin{align*}
\E\bigg[\int_{\R^d} \big|u(t,y)-u_\Dy(t,y)\big|  \,dy \bigg] \le C 
\begin{cases}
\sqrt{\Dy},  & \text{if} \,\,  \lambda<1/2,  \\
\sqrt{\Dy} |\ln \Dy|,  & \text{if}  \,\, \lambda = 1/2, \\
\Dy^{(1-\lambda)}, & \text{if}   \,\,\lambda >1/2,
\end{cases}
\end{align*}
provided the initial error satisfies $\E\Big[ \int_{\R^d} \big|u_{\Dy}(0,y)-u_0(y)\big|
 \,dy\Big] \le \sqrt{\Dy}$.

\section{Numerical Experiments}
\label{numerics}

In this section, we simulate numerical experiments to substantiate the results we have shown in the previous sections. In what follows, inspired by Del Teso \textit{et al.} \cite{DelTeso18,DelTeso19} for the deterministic fractional porous medium operator, we test numerically the performance of the proposed scheme \eqref{scheme}. Here we use a Godunov scheme for the first-order operator and an explicit scheme for the noise term. We set the underlying target equation \eqref{eq:stoc_con_brown} (posed in $\R$) with 
\begin{align*}
A(u)=(u-\frac12)^+,\quad f(u) = \frac12 u^2,\quad \sigma(u)=u(1-u), \,\,\text{and} \,\, u_0(x)=2e^{\frac1{x^2-1}}\mathds{1}_{(-1,1)}. 
\end{align*}
Let us mention  that this configuration of data leads to a solution $u$ satisfying $0 \leq u \leq 1$ so that one may replace $f(u)$ and $\sigma(u)$ by $\frac12 [\min(1,|u|)]^2$ and $u^+(1-u)^+$ respectively to be compatible with the assumptions of the paper.

Next, we present the methodology of the fully discrete explicit numerical scheme. $\Dt>0$ and $\Dx>0$ are the time-step and the spatial mesh size respectively. $t^n =n \Dt$ for $n=0,1,2,\cdots, N= \frac{T}{\Dt}$ denotes the temporal grid and $x_i=i \Dx$ for $i \in \Z$ the spatial one. The scheme becomes
\begin{align}\label{fully_discrete}
U^{n+1}_i= U^n_i- \Dt D^- F(U^n_i,U^n_{i+1})- \frac{\Dt}{\Dx} \sum_{j \in \Z} \widetilde G_{j-i}A(U^n_{j}) + \sigma(U^n_i)\Big[W(t_{(n+1)\Dt})-W(t_{n\Dt})\Big],\\
U^0_i = \frac{1}{\Dx} \int^{x_{i+\frac12}}_{x_{i-\frac12}} u_0(x)\,dx,
\end{align}
where $U^n_i$ is the approximate solution of \eqref{eq:stoc_con_brown} in the cell $[t^n, t^{n+1}) \times [x_{i-\frac12},x_{i+\frac12})$ and we denote the weights by
$\widetilde G_{i-j} := G_{i,j} = G_{j,i}:= \widetilde G_{j-i}$ for $i\neq j$, and $\widetilde G_{0} := G_{i,i}$. The numerical solution is the piecewise constant function denoted by
$$u_\Dx(t,x) = U^n_i, \quad \text{ for all } (t,x) \in [t_n,t_{n+1}) \times [x_{i-\frac{1}{2}},x_{i+\frac12}).$$ A tedious but straightforward calculation reveals the explicit values of the weights:
%
%
%the weights of Section \ref{Finite_Difference_Scheme}:  $\widetilde G_{i} =\widetilde G_{\alpha-\beta} = G_{\alpha,\beta} = G_{\beta,\alpha}= \widetilde G_{\beta-\alpha}=  \widetilde G_{-i}$ for $i=\alpha-\beta$, are given by: 
\[
%\left(
\begin{array}{|c||c|c|}
\hline \displaystyle
\frac{\widetilde{G}_i}{\Dx}  
& \displaystyle
\lambda \neq \frac12  
& \displaystyle
 \lambda = \frac12 \\[0.3cm]
\hline
\hline \displaystyle
|i| \geq 2  
& \displaystyle  
-\frac{d_\lambda(\Dx)^{-2\lambda}}{2\lambda(1-2\lambda)}\Big[
-(i-1)^{1-2\lambda}+2(i)^{1-2\lambda}
-(i+1)^{1-2\lambda} \Big]
& \displaystyle
\frac{-d_\lambda}{\Dx} \ln \Big[\frac{i^2}{i^2-1} \Big]  \\[0.3cm]
\hline \displaystyle
|i|=1  
& \displaystyle
 \frac{d_\lambda \Dx^{-2\lambda}}{2\lambda(1-2\lambda)}\Big[2^{1-2\lambda} +\frac{2\lambda}{2^{1-2\lambda}}-2\Big] 
=
 \frac{d_\lambda \Dx^{-2\lambda}}{2\lambda(1-2\lambda)}\Big[2^{1-2\lambda} +\lambda 2^{2\lambda}-2\Big] 
& \displaystyle  
\frac{-d_\lambda}{\Dx}
\\[0.3cm] 
\hline
i=0 
& \displaystyle  
\frac{d_\lambda(\Dx)^{-2\lambda}}{\lambda(1-2\lambda)}
\bigg(1 -2\lambda (\frac{1}{2})^{1-2\lambda} \bigg)
=
\frac{d_\lambda(\Dx)^{-2\lambda}}{\lambda(1-2\lambda)}
\bigg(1 -\lambda 2^{2\lambda} \bigg)
& \displaystyle  
\frac{d_\lambda}{\Dx}\bigg( 2 + 2\ln2 \bigg)
\\[0.3cm] 
\hline
\end{array}
%\right)
\]
Here $d_\lambda=\frac{2^{2\lambda} \Gamma(\frac{1+2\lambda}{2})}{\pi^{1/2} \Gamma(1-\lambda)}$, with $\Gamma$ being the classical Gamma function.

%\noindent The numerical solution is the piecewise constant function denoted by
%$$u_\Dx(t,x) = U^n_i, \quad \text{ for all } (t,x) \in [t_n,t_{n+1}) \times [x_{i-\frac{1}{2}},x_{i+\frac12}).$$

\subsection{Computation of the non-local diffusion term:} 
Following \cite{Jerome}, a truncated domain $[-K\Dx,K\Dx]$ is considered for a given $K$,  and one considers that $U_i =U_{-K}$ for all $ i \leq -K$ and $U_i= U_K$ for all $i \geq K$. Therefore,  the non-local term in \eqref{fully_discrete} becomes,
\begin{align*}
\frac{\Dt}{\Dx} \sum_{j \in \Z} \widetilde G_{j-i}A(U^n_{j})=
&\frac{\Dt}{\Dx} \sum_{|j| < K} \widetilde G_{j-i}A(U^n_{j}) + \frac{\Dt}{\Dx} A(U^n_{-K}) \sum_{j \leq -K-i} \widetilde G_{j} + \frac{\Dt}{\Dx} A(U^n_{K}) \sum_{j \geq K-i} \widetilde G_{j}
\\ =&
\frac{\Dt}{\Dx} \Big[\sum_{|j| < K} \widetilde G_{j-i} A(U^n_{j}) - \frac12 A(U^n_{-K}) \sum_{|j| < K+i} \widetilde G_j- \frac12 A(U^n_{K})\sum_{|j| < K-i} \widetilde G_j\Big]
\end{align*}
since for any $B>0$,  $\sum_{j \leq -B} \widetilde G_j = \sum_{j \geq B} \widetilde G_j = \frac12 \sum_{|j| \geq B} \widetilde G_j = - \frac12 \sum_{|j| < B} \widetilde G_j$.
%%%%%%%%%%%%%%%%%%%%%%%%

\subsection{Numerical Examples:} 
We have chosen to work with the truncated spatial domain $[-3,3]$ and the time of simulation is $T=1$. We considered as a very small time step $\overline {\Delta t}=2^{-12}$, five times steps $\Delta t \in \{2^{-9},2^{-8},2^{-7},2^{-6}, 2^{-5}\}$ and the corresponding space steps: $\overline{\Delta x}$ and five $\Delta x$, given by using a CFL condition. The later is based on the classical one for monotone flux: $\|f'\|_\infty \frac{\Delta t}{\Delta x}\sim 1$ \cite{Jerome} if $\lambda \leq 0.5$ and on the power $\lambda$ and the weights: $\widetilde{G}_0 \frac{\Delta t}{\Delta x}\sim 1$ \cite{Cifani,Huang} else. 

Since one doesn't know about explicit solutions for such problems, one proposes as a numerical rate of convergence: 
\begin{align*}
\text{Error} := \max_{t \in \{0.25,\ 0.5,\ 0.75,\ 1\}} \widetilde{\E}\Big[\|u_{\Delta x}(t)-u_{\overline{\Delta x}}(t)\|_{L^1(-3,3)}\Big],
\end{align*}
where $\widetilde{\E}$ denotes the statistical average over $5000$ independent paths. 
\\[0.1cm]
This ``Error" is then calculated on the 1088 processing cores research computing cluster ``Pyrene" (univ. Pau) for five values of $\lambda$ in $\{0.8,0.65, 0.5, 0.3,0.1\}$ and the corresponding results are given in the figures below. One can note that the numerical rate of convergence seems to be of order $1$ for smaller values of $\lambda$, and it is of order $1/2$ for larger values of $\lambda$. 
%Thus far from the rather pessimistic lower bound proved in this paper!

\begin{minipage}{19cm}
\begin{minipage}{5cm}
\includegraphics[scale=0.15]{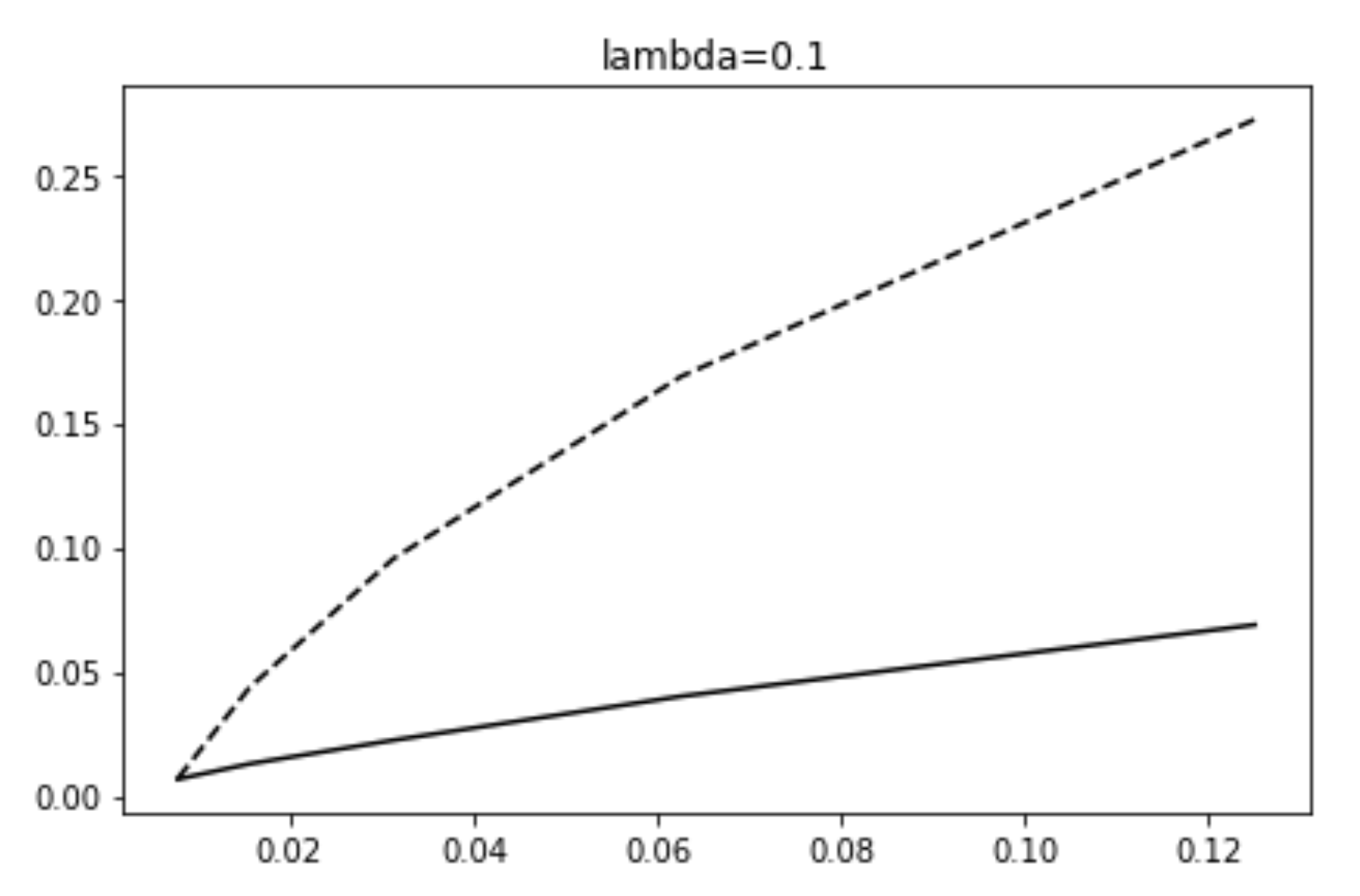}
\end{minipage}
\begin{minipage}{5cm}
\includegraphics[scale=0.15]{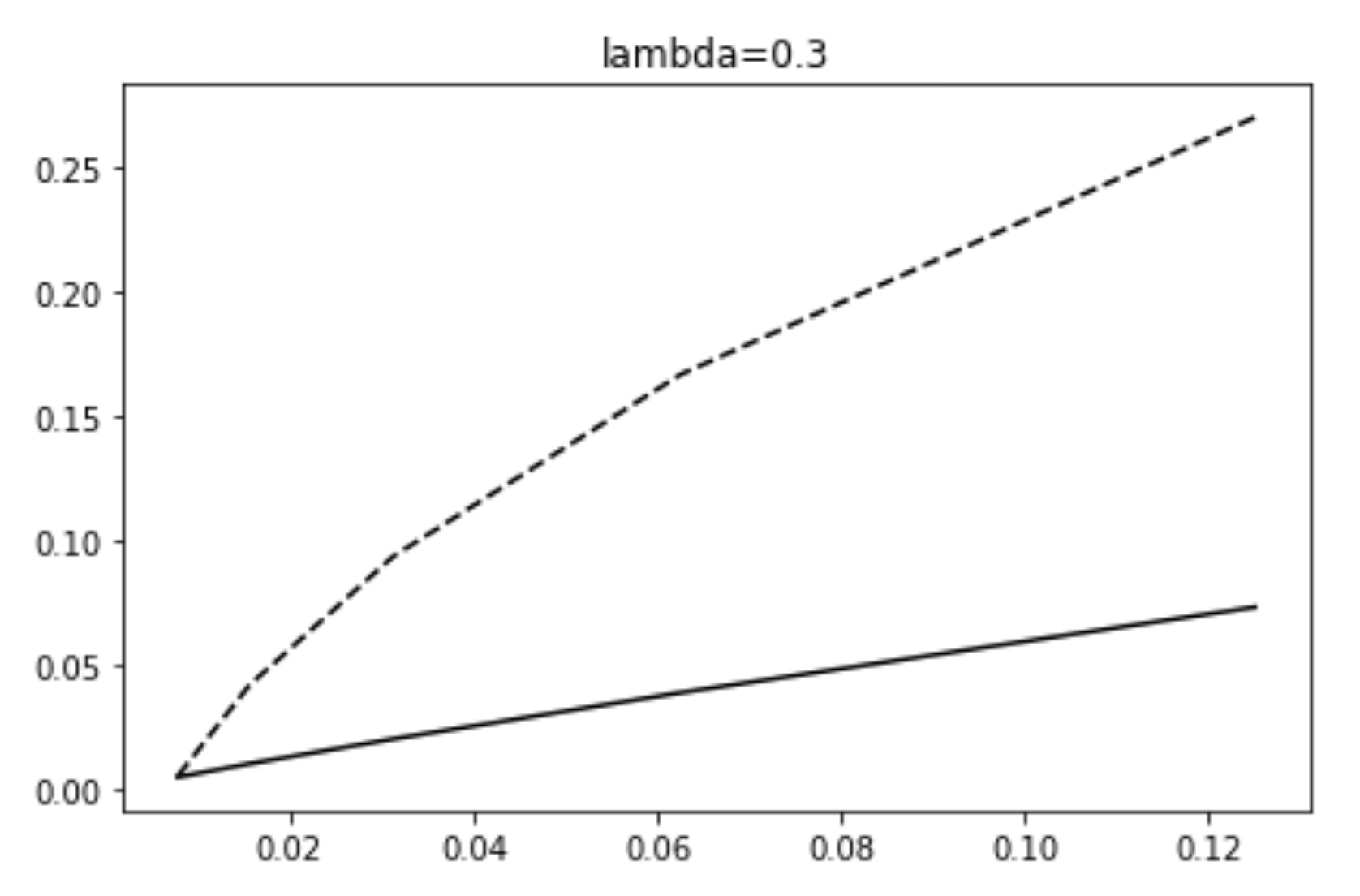}
\end{minipage}
\begin{minipage}{5cm}
\includegraphics[scale=0.15]{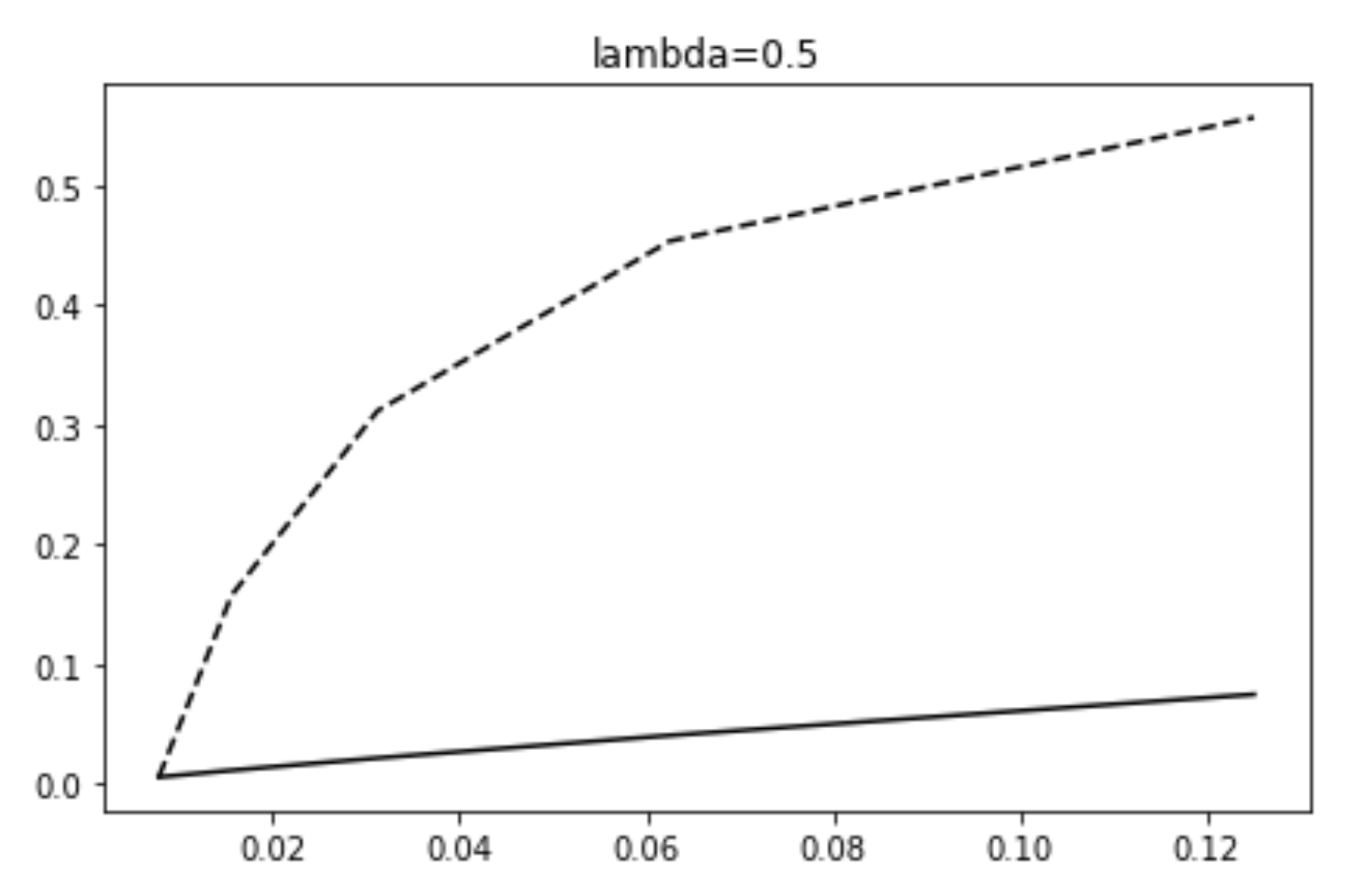}
\end{minipage}

\begin{minipage}{5cm}
\includegraphics[scale=0.15]{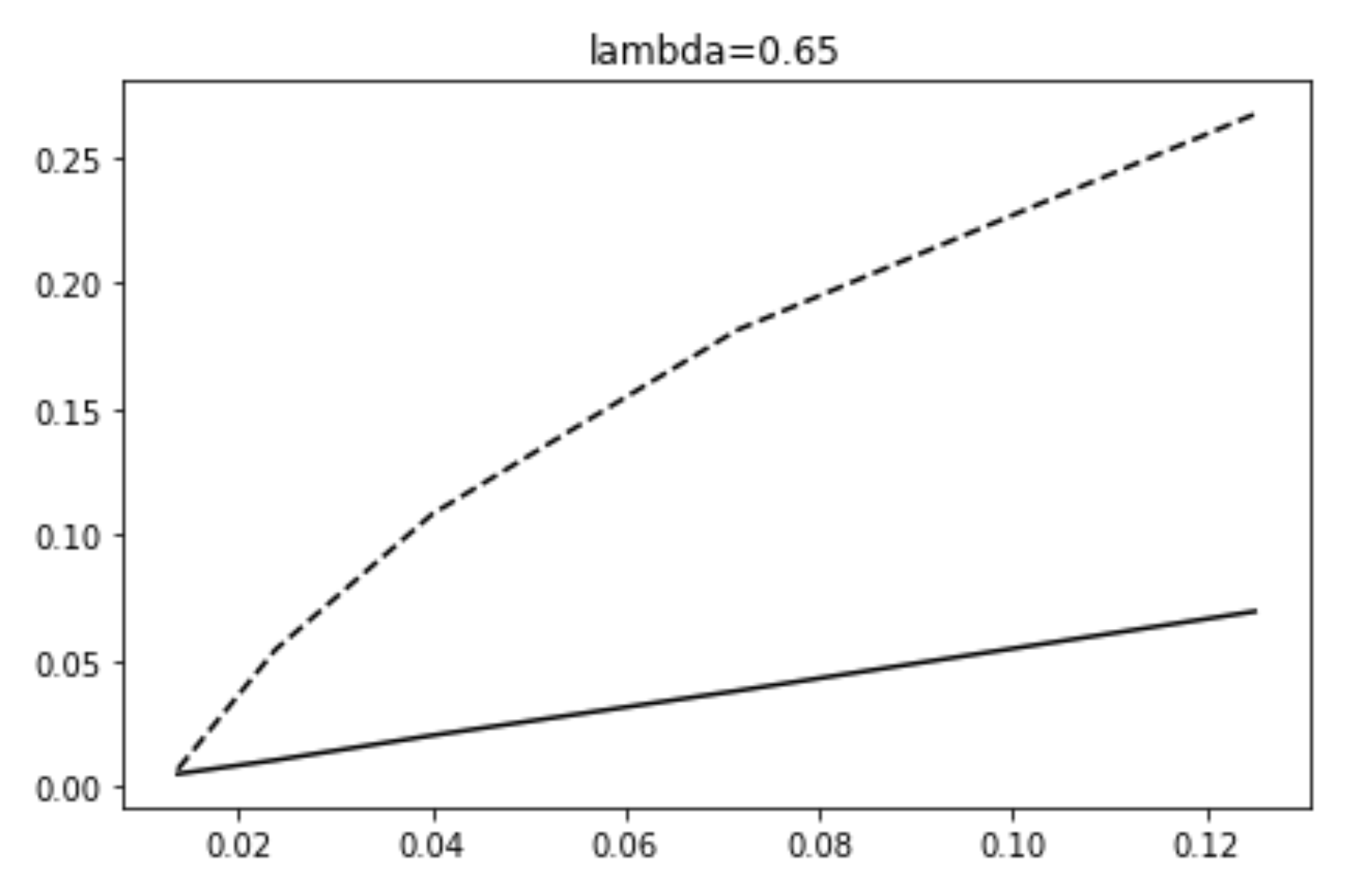}
\end{minipage}
\begin{minipage}{5cm}
\includegraphics[scale=0.15]{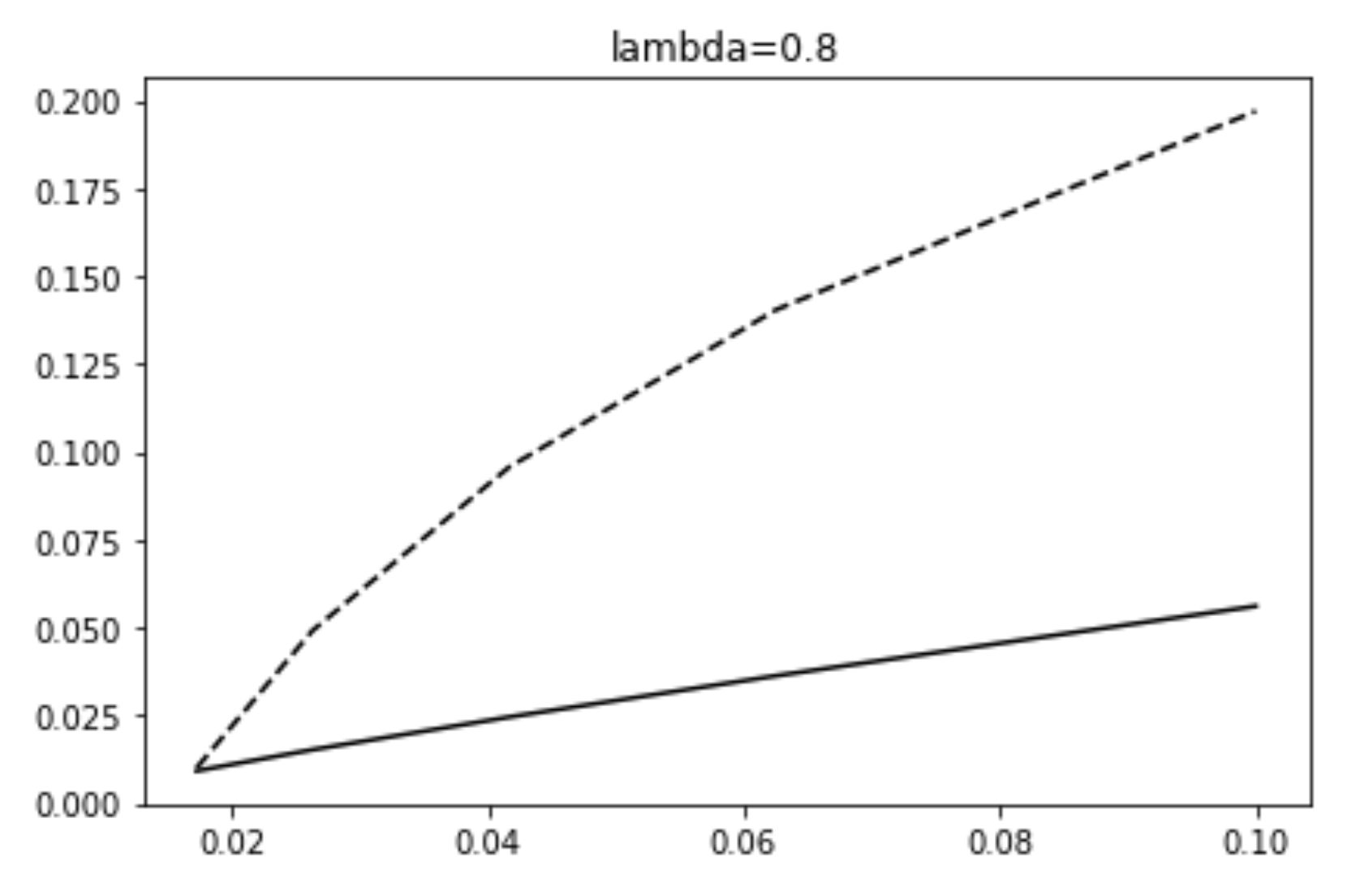}
\end{minipage}
\begin{minipage}{5cm}
Fig. 1: doted curve : curve shape for each $\lambda$ given by Theorem \ref{Main_Theorem} (with $C=1$),
\\
plain curve : curve shape for each $\lambda$ given by "Error".
\end{minipage}
\end{minipage}
\vspace{0.25cm}

{\footnotesize 
\hspace{-0.5cm}
\begin{minipage}[c]{0.34\textwidth}
	\centering
\begin{tabular}{c  c  c c } 
	\toprule
	\bfseries $\Dx$  &\bfseries $\text{Error}$  & \bfseries $\text{Rate}$  \\ 
	\midrule
	\midrule
	0.78 E-2 & 0.52 E-2 & - \\
	\midrule
	1.56 E-2 & 1.07 E-2  & 1.02 \\
	\midrule
	3.12 E-2 & 2.05 E-2 & 0.93 \\
	\midrule
	6.25 E-2 & 3.81 E-2 & 0.89 \\
	\midrule
	12.5 E-2 & 6.97 E-2  & 0.88 \\
	\bottomrule
\end{tabular}
\vspace{0.1cm}
\captionof{table}{$\lambda=0.1$}
\end{minipage}
\hspace{-0.5cm}
\begin{minipage}[c]{0.34\textwidth}
	\centering
	\begin{tabular}{c  c  c c } 
		\toprule
		\bfseries $\Dx$  &\bfseries $\text{Error}$  & \bfseries $\text{Rate}$  \\ 
		\midrule
		\midrule
		0.78 E-2 & 0.53 E-2 & - \\
		\midrule
		1.56 E-2 & 1.09 E-2  & 1.03 \\
		\midrule
		3.12 E-2 & 2.12 E-2 & 0.95 \\
		\midrule
		6.25 E-2 & 4.01 E-2 & 0.92 \\
		\midrule
		12.5 E-2 & 7.44 E-2  & 0.88 \\
		\bottomrule
	\end{tabular}
\vspace{0.1cm}
	\captionof{table}{$\lambda=0.3$}
\end{minipage}
\hspace{-0.5cm}
\begin{minipage}[c]{0.34\textwidth}
	\centering
	\begin{tabular}{c  c  c c } 
		\toprule
		\bfseries $\Dx$  &\bfseries $\text{Error}$  & \bfseries $\text{Rate}$  \\ 
		\midrule
		\midrule
		0.78 E-2 & 0.50 E-2 & - \\
		\midrule
		1.56 E-2 & 1.04 E-2  & 1.05 \\
		\midrule
		3.12 E-2 & 2.04 E-2 & 0.96 \\
		\midrule
		6.25 E-2 & 3.90 E-2 & 0.99 \\
		\midrule
		12.5 E-2 & 7.34 E-2  & 0.90 \\
		\bottomrule
	\end{tabular}
\vspace{0.1cm}
	\captionof{table}{$\lambda=0.5$}
\end{minipage}

\hspace{-0.5cm}
\begin{minipage}[c]{0.35\textwidth}
	\centering
	\begin{tabular}{c  c  c c } 
		\toprule
		\bfseries $\Dx$  &\bfseries $\text{Error}$  & \bfseries $\text{Rate}$  \\ 
		\midrule
		\midrule
		0.78 E-2 & 0.68 E-2 & - \\
		\midrule
		1.56 E-2 & 1.30 E-2  & 0.92 \\
		\midrule
		3.12 E-2 & 2.24 E-2 & 0.79 \\
		\midrule
		6.25 E-2 & 4.01 E-2 & 0.82 \\
		\midrule
		12.5 E-2 & 6.89 E-2  & 0.80 \\
		\bottomrule
	\end{tabular}
\vspace{0.1cm}
	\captionof{table}{$\lambda=0.65$}
\end{minipage}
\hspace{-0.6cm}
\begin{minipage}[c]{0.34\textwidth}
	\centering
	\begin{tabular}{c  c  c c } 
		\toprule
		\bfseries $\Dx$  &\bfseries $\text{Error}$  & \bfseries $\text{Rate}$  \\ 
		\midrule
		\midrule
		0.78 E-2 & 0.92 E-2 & - \\
		\midrule
		1.56 E-2 & 1.52 E-2  & 0.73 \\
		\midrule
		3.12 E-2 & 2.45 E-2 & 0.67 \\
		\midrule
		6.25 E-2 & 3.62 E-2 & 0.56 \\
		\midrule
		12.5 E-2 & 5.12 E-2  & 0.50 \\
		\bottomrule
	\end{tabular}
\vspace{0.1cm}
\captionof{table}{$\lambda=0.8$}
\end{minipage}
\hspace{-0.5cm}
\begin{minipage}{0.34\textwidth}
Tables 1 to 5: Information about the numerical rate of convergence.
\end{minipage}
}

 \end{document}